\documentclass[11pt]{article}

\usepackage{graphicx}
    \graphicspath{ {figures/} }
    \usepackage[margin=1cm]{caption}
    \usepackage{float}
\usepackage{amsmath}
\usepackage{amsthm}
\usepackage{amssymb}
\usepackage[mathscr]{eucal}
\usepackage[all]{xy}
\usepackage{epsfig}
\usepackage{epstopdf}
\usepackage{amscd}
\usepackage[normalem]{ulem}
\usepackage{enumitem}
    \setlist[enumerate,1]{label=\textnormal{(\alph*)}}
    \setlist[enumerate,2]{label=\textnormal{(\roman*)}}

\usepackage[left=1.5in,right=1.5in,top=1.5in,bottom=1.5in,
            ]{geometry}
\usepackage{hyperref}

\theoremstyle{plain}
\newtheorem{theorem}{Theorem}

\newtheorem{corollary}[theorem]{Corollary}

\theoremstyle{remark}

\theoremstyle{definition}

\newtheorem*{question*}{Question}

\newtheorem{observation}{Observation}

\newcommand{\ka}[2]{$#1\textnormal{a}_{#2}$}
\newcommand{\kn}[2]{$#1\textnormal{n}_{#2}$}

\title{Space-Efficient Knot Mosaics for Prime Knots with Mosaic Number 6}
\author{Aaron Heap and Douglas Knowles}

\begin{document}

\maketitle
\begin{abstract} In 2008, Kauffman and Lomonaco introduce the concepts of a knot mosaic and the mosaic number of a knot or link, the smallest integer $n$ such that a knot or link $K$ can be represented on an $n$-mosaic. In \cite{Heap1}, the authors explore space-efficient knot mosaics and the tile number of a knot or link, the smallest number of non-blank tiles necessary to depict the knot or link on a mosaic. They determine bounds for the tile number in terms of the mosaic number. In this paper, we focus specifically on prime knots with mosaic number 6. We determine a complete list of these knots, provide a minimal, space-efficient knot mosaic for each of them, and determine the tile number (or minimal mosaic tile number) of each of them.
\end{abstract}


\section{Introduction}

Mosaic knot theory was first introduced by Kauffman and Lomonaco in the paper \textit{Quantum Knots and Mosaics}
\cite{Lom-Kauff} and was later proven to be equivalent to tame knot theory by Kuriya and Shehab in the paper \textit{The Lomonaco-Kauffman Conjecture} \cite{Kuriya}. The idea of mosaic knot theory is to create a knot or link diagram on an $n \times n$ grid using \textit{mosaic tiles} selected from the collection of eleven tiles shown in Figure \ref{tiles}. The knot or link projection is represented by arcs, line segments, or crossings drawn on each tile. These tiles are identified, respectively, as $T_0$, $T_1$, $T_2$, $\ldots$, $T_{10}$. Tile $T_0$ is a blank tile, and we refer to the rest collectively as non-blank tiles.

\begin{figure}[h]
  \centering
  \includegraphics{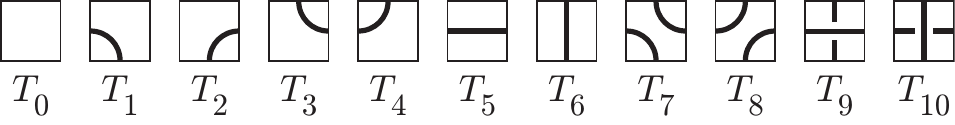}\\
  \caption{Mosaic tiles $T_0$ - $T_{10}$.}
  \label{tiles}
\end{figure}

 A \textit{connection point} of a tile is a midpoint of a tile edge that is also the endpoint of a curve drawn on the tile. A tile is \textit{suitably connected} if each of its connection points touches a connection point of an adjacent tile. An \textit{$n \times n$ knot mosaic}, or \textit{$n$-mosaic}, is an $n \times n$ matrix whose entries are suitably connected mosaic tiles.

\begin{figure}[h]
  \centering
  \includegraphics{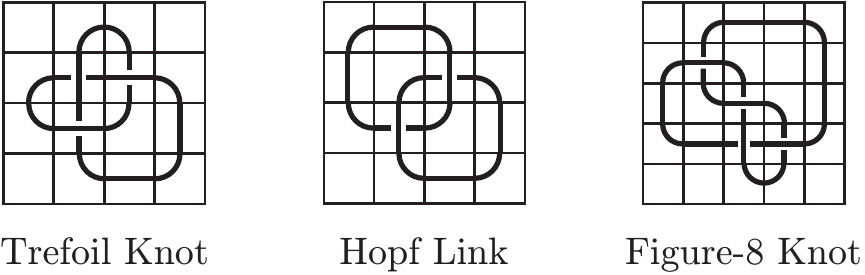}\\
  \caption{Examples of knot mosaics.}
  \label{mosaic-example}
\end{figure}

When listing prime knots with crossing number 10 or less, we will use the Alexander-Briggs notation, matching Rolfsen's table of knots in \emph{Knots and Links} \cite{Rolfsen}. This notation names a knot according to its crossing number with a subscript to denote its order amongst all knots with that crossing number. For example, the $7_4$ knot is the fourth knot with crossing number 7 in Rolfsen's table of knots. For knots with crossing number 11 or higher, we use the Dowker-Thistlethwaite name of the knot. This also names a knot according to its crossing number, with an ``a'' or ``n'' to distinguish the alternating and non-alternating knots and a subscript that denotes the lexicographical ordering of the minimal Dowker-Thistlethwaite notation for the knot. For example \ka{11}{7} is the seventh alternating knot with crossing number 11, and \kn{11}{3} is the third non-alternating knot with crossing number 11. For more details on these and other relevant information on traditional knot theory, we refer the reader to \emph{The Knot Book} by Adams \cite{Adams}.

The \textit{mosaic number} of a knot or link $K$ is the smallest integer $n$ for which $K$ can be represented as an $n$-mosaic. We denote the mosaic number of $K$ as $m(K)$. The mosaic number has previously been determined for every prime knot with crossing number 8 or less. For details, see \textit{Knot Mosaic Tabulations} \cite{Lee2} by Lee, Ludwig, Paat, and Peiffer. In particular, it is known that the unknot has mosaic number 2, the trefoil knot has mosaic number 4, the $4_1$, $5_1$, $5_2$, $6_1$, $6_2$, and $7_4$ knots have mosaic number 5, and all other prime knots with crossing number eight or less have mosaic number 6. In this paper, we determine the rest of the prime knots that have mosaic number 6, which includes prime knots with crossing numbers from 9 up to 13. This confirms, in the case where the mosaic number is $m=6$, a result of Howards and Kobin in \cite{Howards}, where they find that the crossing number is bounded above by $(m-2)^2 -2$ if $m$ is odd, and by $(m-2)^2 -(m-3)$ if $m$ is even. We also determine that not all knots with crossing number 9 (or more) have mosaic number 6.

Another number associated to a knot mosaic is the \emph{tile number of a mosaic}, which is the number of non-blank tiles used to create the mosaic. From this we get an invariant called the \emph{tile number $t(K)$ of a knot or link} $K$, which is the least number of non-blank tiles needed to construct $K$ on a mosaic of any size. In \cite{Heap1}, the authors explore the tile number of a knot or link and determine strict bounds for the tile number of a prime knot $K$ in terms of the mosaic number $m \geq 4$. Specifically, if $m$ is even, then $5m-8 \leq t(K) \leq m^2-4$. If $m$ is odd, then $5m-8 \leq t(K) \leq m^2-8$. It follows immediately that the tile number of the trefoil knot must be 12, and the tile number of the prime knots mentioned above with mosaic number 5 must be 17. The authors also listed several prime knots with mosaic number 6 that have have the smallest possible tile number $t(K)=22$, which we summarize in Theorem \ref{known-tile-numbers}. In this paper, we confirm that this list is complete.

\begin{theorem}[{\cite{Heap1}}]\label{known-tile-numbers} The following knots have the given tile number.
 \begin{enumerate}
  \item Tile number 4: Unknot.
  \item Tile number 12: Trefoil knot.
  \item Tile number 17: $4_1$, $5_1$, $5_2$, $6_1$, $6_2$, and $7_4$.
  \item Tile number 22:
   \begin{enumerate}
    \item $6_3$,
    \item $7_1$, $7_2$, $7_3$, $7_5$, $7_6$, $7_7$,
    \item $8_1$, $8_2$, $8_3$, $8_4$, $8_7$, $8_8$, $8_9$, $8_{13}$,
    \item $9_5$, and $9_{20}$.
   \end{enumerate}
 \end{enumerate}
\end{theorem}

Knot mosaics in which the tile number is realized for each of these mosaics are given in \cite{Heap1} and also in this paper in the table of mosaics in Section \ref{table_of_knots}. Finally, in \cite{Heap1}, the authors determine all of the possible layouts for any prime knot on an $n$-mosaic, for $n \leq 6$. In this paper, we complete that work by determining which prime knots can be created from those layouts.

We also point out that throughout this paper we make significant use of the software package KnotScape \cite{Thistle}, created by Thistlethwaite and Hoste, to verify that a given knot mosaic represents a specific knot. Without this program, the authors of this paper would not have been able to complete the work.

\section{Space-Efficient Knot Mosaics}

Two knot mosaic diagrams are of the \emph{same knot type} (or \emph{equivalent}) if we can change one to the other via a sequence of \textit{mosaic planar isotopy moves} that are analogous to the planar isotopy moves for standard knot diagrams. An example of this is shown in Figure \ref{isotopy}. A complete list of all of these moves are given and discussed in \cite{Lom-Kauff} and \cite{Kuriya}. We will make significant use of these moves throughout this paper, as we attempt to reduce the tile number of mosaics in order to construct knot mosaics that use the least number of non-blank tiles.

\begin{figure}[h]
  \centering
  \includegraphics{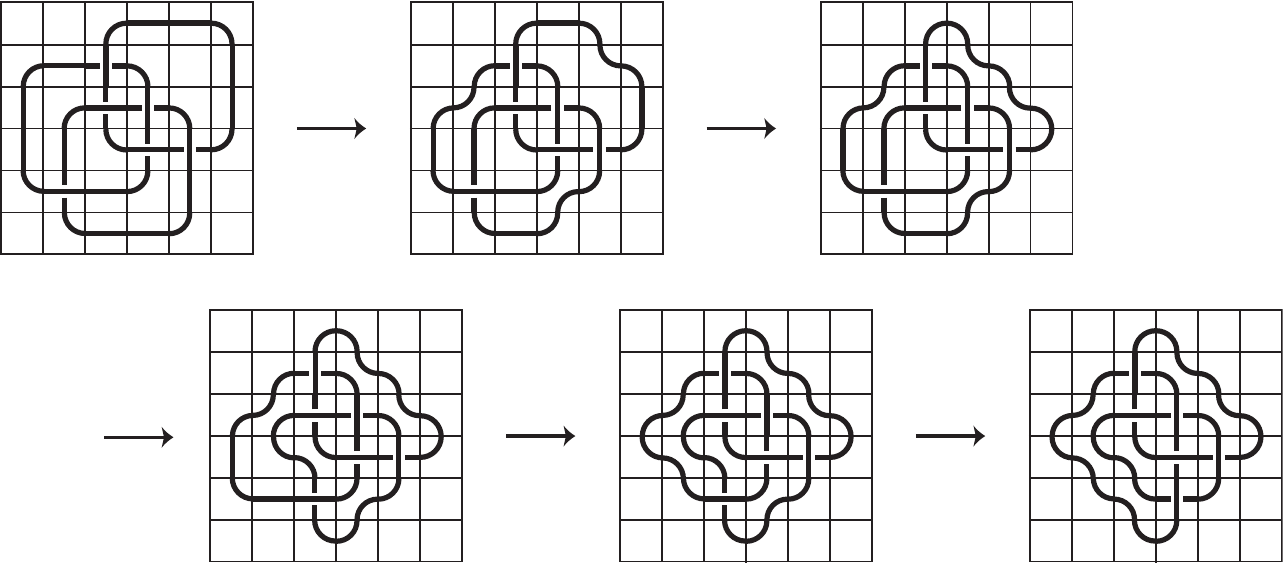}\\
  \caption{Example of mosaic planar isotopy moves.}
  \label{isotopy}
\end{figure}

A knot mosaic is called \textit{minimal} if it is a realization of the mosaic number of the knot depicted on it. That is, if a knot with mosaic number $m$ is depicted on an $m$-mosaic, then it is a minimal knot mosaic. A knot mosaic is called \textit{reduced} if there are no unnecessary, reducible crossings in the knot mosaic diagram. That is, we cannot draw a simple, closed curve on the knot mosaic that intersects the knot diagram transversely at a single crossing but does not intersect the knot diagram at any other point. (See \cite{Adams} for more on reduced knot diagrams.)

We have already defined the tile number of a mosaic and the tile number of a knot or link. A third type of tile number is the \textit{minimal mosaic tile number $t_M(K)$ of a knot or link} $K$, which is the smallest number of non-blank tiles needed to construct $K$ on a minimal mosaic. That is, it is the smallest possible tile number of all possible minimal mosaic diagrams for $K$. That is, for a knot with mosaic number $m$, the minimal mosaic tile number of the knot is the least number of non-blank tiles needed to construct the knot on an $m \times m$ mosaic. Much like the crossing number of a knot cannot always be realized on a minimal mosaic (such as the $6_1$ knot), the tile number of a knot cannot always be realized on a minimal mosaic. Note that the tile number of a knot or link $K$ is certainly less than or equal to the minimal mosaic tile number of $K$, $t(K) \leq t_M(K)$. The fact that the tile number of a knot is not necessarily equal to the minimal mosaic tile number of the knot is confirmed later in Theorem \ref{min-mos-tile-number}. However, for prime knots, it is shown in \cite{Heap1} that $t_M(K)=t(K)$ when $t_M(K) \leq 27$.

A knot $n$-mosaic is \textsl{space-efficient} if it is reduced and the tile number is as small as possible on an $n$-mosaic without changing the knot type of the depicted knot, meaning that the tile number cannot be decreased through a sequence of mosaic planar isotopy moves. A knot mosaic is \textsl{minimally space-efficient} if it is minimal and space-efficient. The first four knot mosaics of the Borromean rings depicted in Figure \ref{isotopy} are not space-efficient because we can decrease the tile number through the depicted mosaic planar isotopy moves. In Figure \ref{space-eff-5-1}, both mosaics are knot mosaic diagrams of the $5_1$ knot. The first knot mosaic is not space-efficient, but the second knot mosaic is minimally space-efficient.

\begin{figure}[h]
  \centering
  \includegraphics{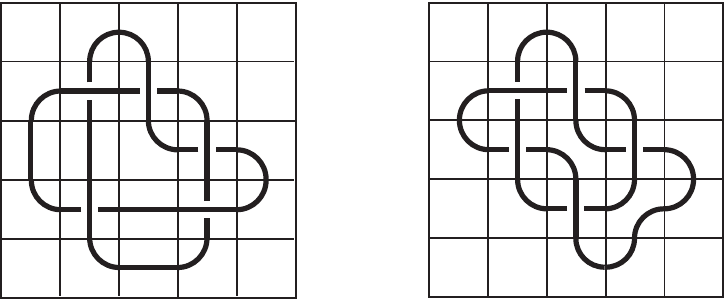}\\
  \caption{Space-inefficient and minimally space-efficient knot mosaics of $5_1$.}
  \label{space-eff-5-1}
\end{figure}

On a minimally space-efficient knot mosaic, the minimal mosaic tile number of the depicted knot must be realized, but the tile number of the knot might not be realized. There may be a larger, non-minimal knot mosaic that uses fewer non-blank tiles, meaning that a space-efficient knot mosaic need not be minimally space-efficient.

In addition to the original eleven tiles $T_0$ - $T_{10}$, we will also make use of \textit{nondeterministic tiles}, such as those in Figure \ref{nondeterministic}, when there are multiple options for the tiles that can be placed in specific tile locations of a mosaic. For example, if a tile location must contain a crossing tile $T_9$ or $T_{10}$ but we have not yet chosen which, we will use the nondeterministic crossing tile. Similarly, if we know that a tile location must have four connection points but we do not know if the tile is a double arc tile ($T_7$ or $T_{8}$) or a crossing tile ($T_9$ or $T_{10}$), we will indicate this with a tile that has four connection points.

\begin{figure}[h]
  \centering
  \includegraphics{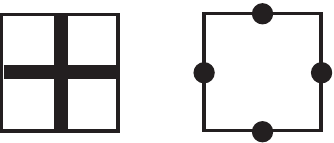}\\
  \caption{Nondeterministic crossing tile and a nondeterministic tile with four connection points.}
  \label{nondeterministic}
\end{figure}

\section{Minimally Space-Efficient 6-Mosaics of Prime Knots}\label{knot-search}

In \cite{Heap1}, the authors provide the possible tile numbers (and the layouts that result in these tile numbers) for all prime knots on a space-efficient 6-mosaic.

\begin{theorem}[{\cite{Heap1}}]\label{tile-numbers} If we have a space-efficient 6-mosaic of a prime knot $K$ for which either every column or every row is occupied, then the only possible values for the tile number of the mosaic are 22, 24, 27, and 32. Furthermore, any such mosaic of $K$ is equivalent (up to symmetry) to one of the mosaics in Figure \ref{tile-numbers-1}.
\begin{figure}[h]
  \centering
  \includegraphics{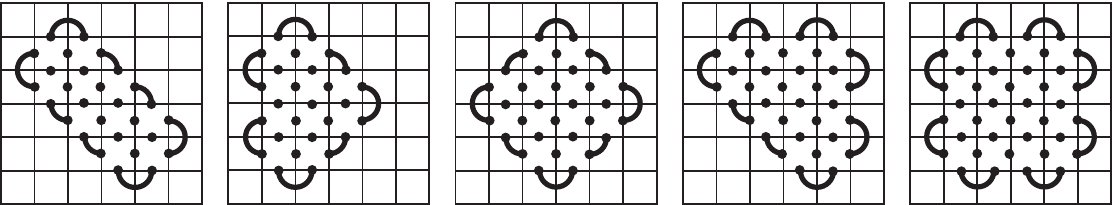}\\
  \caption{Only possible layouts for a space-efficient 6-mosaic.}
  \label{tile-numbers-1}
\end{figure}
\end{theorem}

In order to determine the tile number or minimal mosaic tile number of prime knots with mosaic number 6, we need to determine which prime knots can be written as a knot mosaic with these layouts. We not only want to know which ones can be expressed, for example, on the layouts with twenty-two non-blank tiles, but we also want to know which ones cannot be expressed in this way. If a knot cannot be expressed on the layout with twenty-two non-blank tiles, then it must have tile number larger than 22. We will do this for each of the layouts in Figure \ref{tile-numbers-1}. To help us with this, we make a few simple observations. All of these are easy to verify, and any rotation or reflection of these scenarios are also valid.

In order to easily refer to specific tile locations within a mosaic, on a $6 \times 6$ mosaic we label all of the boundary tiles except the corner tiles as $B^1$ - $B^{16}$, and we label the sixteen tiles on the inner board as $I^1$ - $I^{16}$.
\begin{figure}[h]
  \centering
  \includegraphics{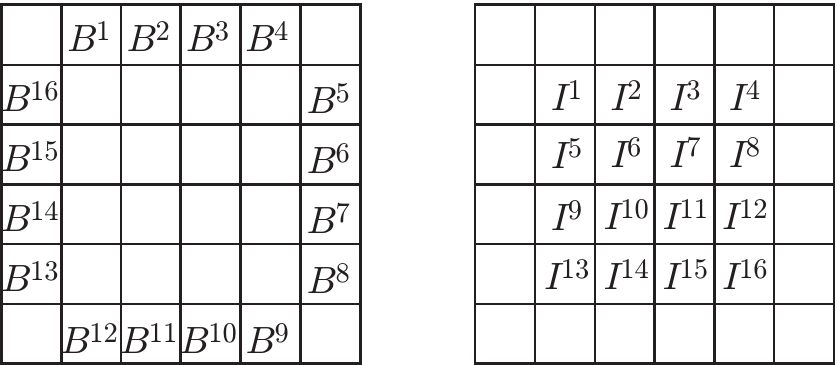}\\
  \caption{Labels of the tiles in a $6 \times 6$ mosaic board.}
  \label{grid-labels}
\end{figure}

Consider the upper, right $3 \times 3$ corner of any space-efficient mosaic of a prime knot with mosaic number 6 and tile number 22, 27, or 32. (That is, we are considering every option except those with tile number 24.) It must be one of the two options in Figure \ref{3x3corners}. All other $3 \times 3$ corners are a rotation of one of these. We will refer to the first option as a \textit{partially filled block} and the second option as a \textit{filled block}.

\begin{figure}[h]
  \centering
  \includegraphics{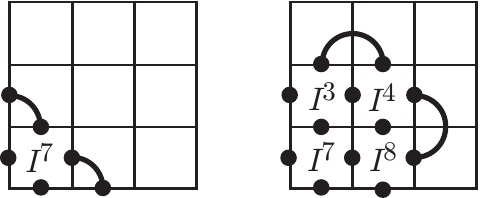}\\
  \caption{A partially filled block and a filled block, respectively.}
  \label{3x3corners}
\end{figure}

\begin{observation} In any space-efficient 6-mosaic of a prime knot, the tile in position $I^7$ of a partially filled block is either a crossing tile or double arc $T_7$.
\end{observation}

This is easy to see, as it must be a tile with four connection points, and the only space-efficient mosaics that results from using the double arc $T_8$ are composite knots or links with more than one component. In Figure \ref{observ2}, the first two examples are valid possibilities, but the third one is not.

\begin{figure}[h]
  \centering
  \includegraphics{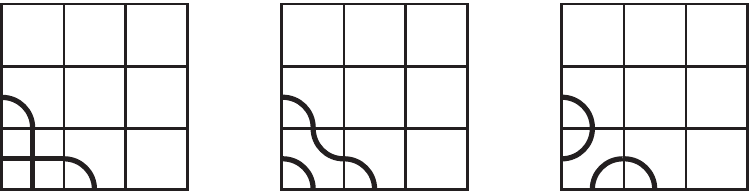}\\
  \caption{The first two examples are the only valid possibilities for a partially filled block, but the third one is not.}
  \label{observ2}
\end{figure}

\begin{observation} In any space-efficient 6-mosaic of a prime knot, there must be at least two crossing tiles in a filled block.
\end{observation}

\begin{figure}[h]
  \centering
  \includegraphics{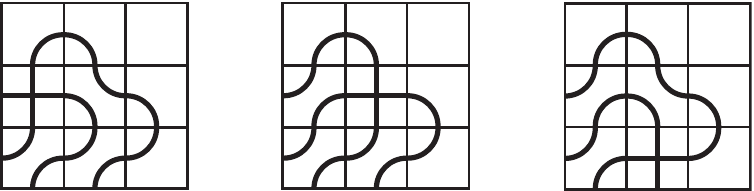}\\
  \caption{Suitably connected filled blocks with one crossing in position $I^{3}$, $I^{4}$, or $I^{8}$. None are space-efficient.}
  \label{observ3}
\end{figure}

If there are no crossing tiles in positions $I^{3}$, $I^{4}$, $I^{7}$, and $I^{8}$ of the mosaic, then the mosaic is not space-efficient or it is a link with more than one component. Each one that is not a link reduces to one of the last two partially filled block options in Figure \ref{observ2}. If there is only one crossing tile and it is in position $I^{3}$, $I^{4}$, or $I^{8}$, then the mosaic is not space-efficient. For each option, if we fill the remaining tile positions with double arc tiles so that the block is suitably connected and we avoid the obvious inefficiencies we get the options shown in Figure \ref{observ3}. They are equivalent to each other via a simple mosaic planar isotopy move that rolls the crossing through each of these positions, and they all reduce to the first partially filled block in Figure \ref{observ2}. If there is only one crossing tile and it is in position $I^{7}$, then the mosaic is also not space-efficient and reduces to either of the first two options in Figure \ref{observ2}.

\begin{observation} In a filled block in any space-efficient 6-mosaic of a prime knot, there are only two distinct possibilities for two crossing tiles, two distinct possibilities for three crossing tiles, and one possibility for four crossing tiles, and they are shown in Figure \ref{observ4}.
\end{observation}

\begin{figure}[h]
  \centering
  \includegraphics{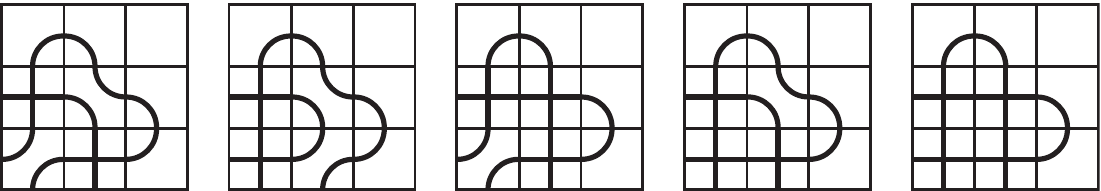}\\
  \caption{The only possible layouts for a filled block.}
  \label{observ4}
\end{figure}

We will refer to the five filled blocks in Figure \ref{observ4} together with the first two partially filled blocks in Figure \ref{observ2} (and reflections and rotations of them) as \textit{building blocks}. The observations provide a way for us to easily build all of the space-efficient $6$-mosaics, as long as the tile number is 22, 27, or 32, but not 24.

\begin{observation}\label{observation4} In any space-efficient 6-mosaic of a prime knot, there is at most one of a filled block with four crossing tiles or a filled block with two crossings in position $I^3$ and $I^7$.
\end{observation}

It is quite simple to verify that if there is more than one filled block with four crossings or more than one filled block with two crossings in positions $I^3$ and $I^7$, the resulting mosaic must be a link with more than one component. If we use the indicated filled building block with two crossing tiles together with a filled block with four crossing tiles, the resulting mosaic will also be a link with more than one component. Several examples of these are pictured in Figure \ref{observ5-6} with the second link component in each mosaic colored differently from the first link component.

\begin{figure}[h]
  \centering
  \includegraphics{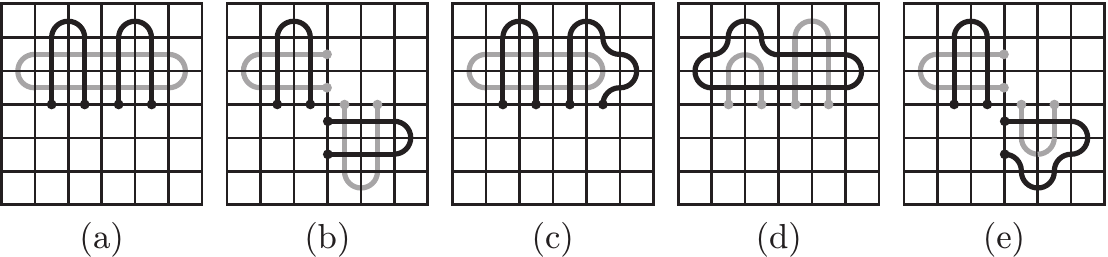}\\
  \caption{These layouts will always be multi-component links.}
  \label{observ5-6}
\end{figure}

We are now ready to determine the tile number of every prime knot with mosaic number 6. In order to do this, we simply compile a list of the prime knots that can fit within each of the layouts given in Figure \ref{tile-numbers-1}. Because we already know the tile number of every prime knot with crossing number 7 or less, we can restrict our search to knots with crossing number 8 or more. The process is simple, and the above observations help us tremendously. If the tile number is 22, 27, or 32, we use the building blocks. In the case of the mosaics with tile number 24, we look at all possible placements, up to symmetry, of eight or more crossing tiles within the mosaics and fill the remaining tile positions with double arc tiles so as to avoid composite knots and non-reduced knots. Once the mosaics are completed, we then eliminate any links, any duplicate layouts that are equivalent to others via obvious mosaic planar isotopy moves, and any mosaics for which the tile number can easily be reduced by a simple mosaic planar isotopy move. Finally, we use KnotScape to determine what knots are depicted in the mosaic be choosing the crossings so that they are alternating, as well as all possible non-alternating combinations. We provide minimally space-efficient knot mosaics for every prime knot with mosaic number less than or equal to 6 in the table of knots in Section \ref{table_of_knots}.

We have already listed several prime knots with tile number 22 in Theorem \ref{known-tile-numbers}. This next theorem asserts that the list is complete.

\begin{theorem}\label{t=22}
The only prime knots $K$ with 
tile number $t(K)=22$ are:
\begin{enumerate}
  \item $6_{3}$,
  \item $7_1$, $7_2$, $7_3$, $7_5$, $7_6$, $7_7$,
  \item $8_1$, $8_2$, $8_3$, $8_4$, $8_7$, $8_8$, $8_9$, $8_{13}$,
  \item $9_5$, and $9_{20}$.
\end{enumerate}
\end{theorem}

In order to obtain the minimally space-efficient knot mosaic for $7_3$, we had to use eight crossings. None of the possible minimally space-efficient knot mosaics with twenty-two non-blank tiles and exactly seven crossings produced $7_3$. The fewest number of non-blank tiles needed to represent $7_3$ with only seven crossings is twenty-four, and one such mosaic is given in Figure \ref{7_3-7and8}, along with a minimally space-efficient mosaic of $7_3$ with eight crossings. In summary, on a minimally space-efficient knot mosaic, for the tile number (or minimal mosaic tile number) to be realized, it might not be possible for the crossing number to be realized. This is also the case with $8_1$, $8_3$, $8_7$, $8_8$, and $8_9$, as nine crossing tiles are required to represent these knots on a mosaic with tile number 22.

\begin{figure}[h]
  \centering
  \includegraphics{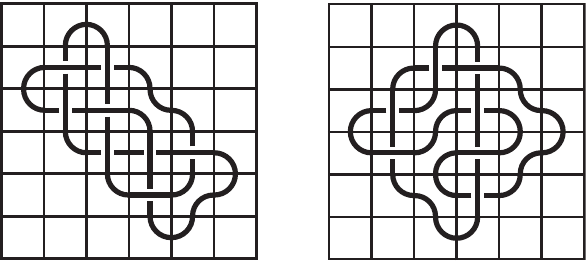}\\
  \caption{The $7_3$ knot as a minimally space-efficient knot mosaic with eight crossing tiles and as a knot mosaic with seven crossing tiles.}
  \label{7_3-7and8}
\end{figure}

\begin{proof} We simply build the first two tile configurations (both with twenty-two non-blank tiles) in Figure \ref{tile-numbers-1} using the $3 \times 3$ building blocks, eliminate any that do not satisfy the observations, choose specific crossing types, and see what we get. Whatever prime knots with eight or more crossings are missing are the ones we know cannot have tile number 22.

We begin with the first mosaic layout given in Figure \ref{tile-numbers-1}. Up to symmetry, there are only six possible configurations of this layout with eight crossings, and they are given in Figure \ref{22tiles-8crossings-1}. Notice that some of these are links that can be eliminated, including Figures \ref{22tiles-8crossings-1}(d) and (f). Furthermore, Figures \ref{22tiles-8crossings-1}(b) and (c) are equivalent to each other via a mosaic planar isotopy move that shifts one of the crossing tiles to a diagonally adjacent tile position. This leaves us with only three possible distinct configurations of eight crossings from this first layout, Figures \ref{22tiles-8crossings-1}(a), (b), and (e).

\begin{figure}[h]
  \centering
  \includegraphics{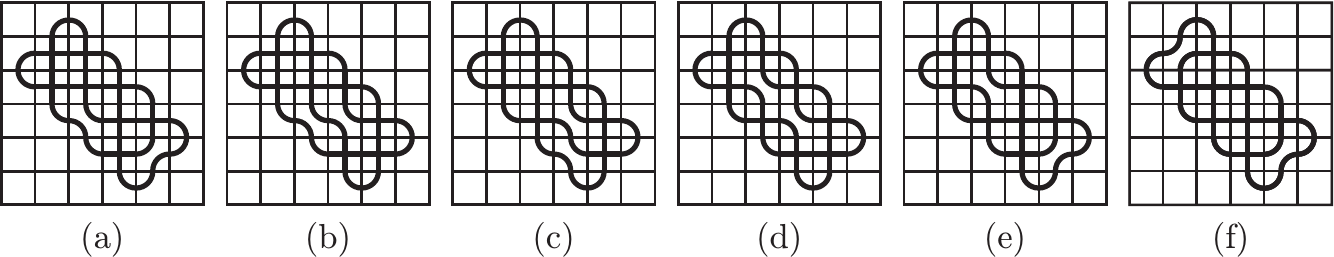}\\
  \caption{Possible placements of eight crossing tiles in the first layout with tile number 22.}
  \label{22tiles-8crossings-1}
\end{figure}

Now we do the same thing with the second mosaic layout given in Figure \ref{tile-numbers-1} with twenty-two non-blank tiles. Up to symmetry, there are six possible configurations of this layout with eight crossings, and they are given in Figure \ref{22tiles-8crossings-2}. Again, Figures \ref{22tiles-8crossings-2}(d) and (f) are links, and Figures \ref{22tiles-8crossings-2}(b) and (c) are equivalent to each other. This leaves us again with only three possible configurations of eight crossings from this second layout, and they are again Figures \ref{22tiles-8crossings-2}(a), (b), and (e). Moreover, each one of these is equivalent to the corresponding mosaics in Figure \ref{22tiles-8crossings-1} via a few mosaic planar isotopy moves that shift the crossings in the lower, left building block into the lower, right building block of the mosaic.

\begin{figure}[h]
  \centering
  \includegraphics{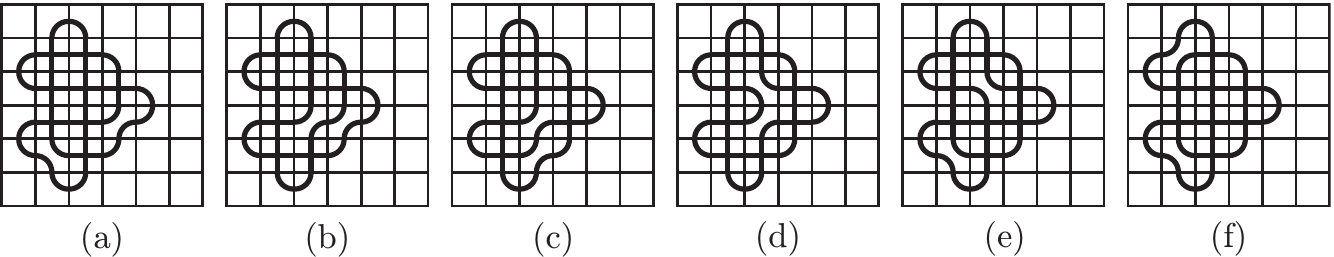}\\
  \caption{Possible placements of eight crossing tiles in the second layout with tile number 22.}
  \label{22tiles-8crossings-2}
\end{figure}

This leaves us with only three distinct possible layouts for a minimally space-efficient $6 \times 6$ mosaic with eight crossings and tile number 22. If we choose crossings for the configuration in Figure \ref{22tiles-8crossings-1}(a) so that they are alternating, we get the $8_{13}$ knot. If we choose crossings for the configuration in Figure \ref{22tiles-8crossings-1}(b) so that they are alternating, we get the $8_{4}$ knot. Finally, if we choose crossings for the configuration in Figure \ref{22tiles-8crossings-1}(e) so that they are alternating, we get the $8_{2}$ knot. If we examine all possible non-alternating choices for each one, all of the resulting knots have crossing number seven or less. (The minimally space-efficient knot mosaic for $7_3$ must have eight crossing tiles and can be obtained by a choice of non-alternating crossings within any of the three distinct possible layouts in Figure \ref{22tiles-8crossings-1}.)

Now we go through the same process using nine crossing tiles. Up to symmetry, there are only four possible configurations of these layouts with nine crossings, and they are given in Figure \ref{22tiles-9or10crossings-1}. The mosaic in Figure \ref{22tiles-9or10crossings-1}(c) is equivalent to the mosaic in Figure \ref{22tiles-9or10crossings-1}(b) via a few mosaic planar isotopy moves that shifts the crossings in the lower, left building block into the lower, right building block of the mosaic. This leaves us with only three possible configurations of nine crossing tiles.

\begin{figure}[h]
  \centering
  \includegraphics{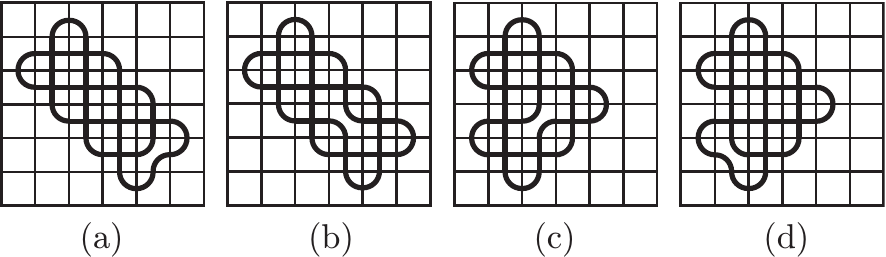}\\
  \caption{Possible placements of nine crossings with tile number 22.}
  \label{22tiles-9or10crossings-1}
\end{figure}

If we choose crossings for the configuration in Figure \ref{22tiles-9or10crossings-1}(a) so that they are alternating, we get the $9_{20}$ knot. If we examine all possible non-alternating choices for the crossings, most of the resulting knots have crossing number seven or less, but we do get some additions to our list of prime knots with tile number 22 and crossing number 8. In particular, we get $8_7$, $8_8$, and $8_9$. (We also get $8_4$, which was previously obtained with only eight crossings.) If we choose crossings for the configuration in Figure \ref{22tiles-9or10crossings-1}(b) so that they are alternating, we get the $9_{5}$ knot. Again, if we examine the possible non-alternating choices for the crossings, we get two additional prime knots with tile number 22 and crossing number 8, and they are $8_1$ and $8_3$. Finally, if we choose crossings for the configuration in Figure \ref{22tiles-9or10crossings-1}(d), we get the exact same knots as we did for Figure \ref{22tiles-9or10crossings-1}(a).

By Observation \ref{observation4}, we cannot place more than nine crossing tiles on any mosaic with twenty-two non-blank tiles. We have now found every possible prime knot with tile number 22 and eight or more crossings, and they are exactly those listed in the theorem. All other prime knots with crossing number at least eight must have tile number larger than 22.
\end{proof}

We now know precisely which prime knots have tile number 22 or less. Our next goal is to determine which prime knots have tile number 24.

\begin{theorem}\label{t=24}
The only prime knots $K$ with 
tile number $t(K)=24$ are:
\begin{enumerate}
  \item $8_5$, $8_6$, $8_{10}$, $8_{11}$, $8_{12}$, $8_{14}$, $8_{16}$, $8_{17}$, $8_{18}$, $8_{19}$, $8_{20}$, $8_{21}$,
  \item $9_8$, $9_{11}$, $9_{12}$, $9_{14}$, $9_{17}$, $9_{19}$, $9_{21}$, $9_{23}$, $9_{26}$, $9_{27}$, $9_{31}$,
  \item $10_{41}$, $10_{44}$, $10_{85}$, $10_{100}$, $10_{116}$, $10_{124}$, $10_{125}$, $10_{126}$, $10_{127}$, $10_{141}$, $10_{143}$, $10_{148}$, $10_{155}$ and $10_{159}$.
\end{enumerate}
\end{theorem}

 We will show that $8_6$ must have nine crossing tiles to fit on a mosaic with tile number 24. None of the possible minimally space-efficient knot mosaics with exactly eight crossings produce these knots. Similarly, the minimally space-efficient mosaics for $9_{12}$, $9_{19}$, $9_{21}$, and $9_{26}$ require ten crossings.

\begin{proof}  We search for all of the prime knots that have  tile number 24. In this particular case, the observations at the beginning of this section do not apply, meaning we cannot use the building blocks as we did in the proof of Theorem \ref{t=22}. We know from Theorem \ref{tile-numbers} that any prime knot with tile number 24 has a space-efficient mosaic like the third layout in Figure \ref{tile-numbers-1}. We simply look at all possible placements of eight or more crossings within that layout, choose the type of each crossing, and keep track of the resulting prime knots.

First, we look at all possible placements, up to symmetry, of eight crossings within the mosaic and, we fill the remaining tile positions with double arc tiles so as to avoid composite knots and unnecessary loops. After eliminating any links and any duplicate layouts that are equivalent to others via simple mosaic planar isotopy moves, we get seventeen possible layouts, which are shown in Figure \ref{24tiles-8crossings-1}. Not all of these will result in distinct knots, and in most cases it is not difficult to see that they will result in the same knot. However, we include all of them here because they differ by more than just simple symmetries or simple mosaic planar isotopy moves.

\begin{figure}[h]
  \centering
  \includegraphics{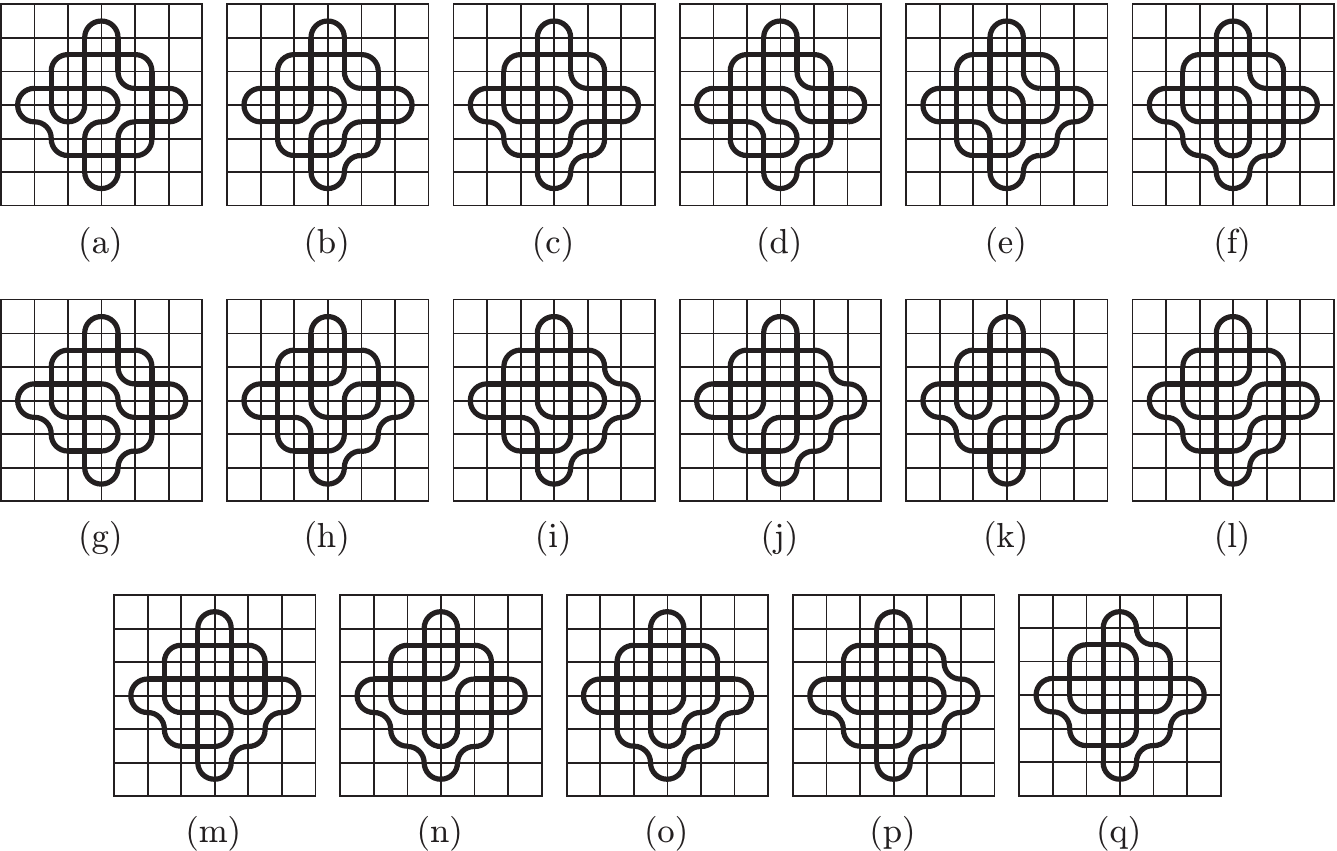}\\
  \caption{Only possible layouts, after elimination, with eight crossing tiles for a prime knot with  tile number 24.}
  \label{24tiles-8crossings-1}
\end{figure}

Choosing specific crossings so that the knots are alternating, we obtain only fourteen distinct knots as follows. Figure \ref{24tiles-8crossings-1}(a) is the $8_{1}$ knot. Figures \ref{24tiles-8crossings-1}(b) and (c) are $8_{2}$. Figure \ref{24tiles-8crossings-1}(d) is $8_{4}$. Figure \ref{24tiles-8crossings-1}(e) is $8_{5}$. Figures \ref{24tiles-8crossings-1}(f) and (g) are $8_{7}$. Figure \ref{24tiles-8crossings-1}(h) is $8_{8}$. Figure \ref{24tiles-8crossings-1}(i) is $8_{10}$. Figure \ref{24tiles-8crossings-1}(j) is $8_{11}$. Figure \ref{24tiles-8crossings-1}(k) is $8_{12}$. Figure \ref{24tiles-8crossings-1}(l) is $8_{13}$. Figures \ref{24tiles-8crossings-1}(m) and (n) are $8_{14}$. Figure \ref{24tiles-8crossings-1}(o) is $8_{16}$. Figure \ref{24tiles-8crossings-1}(p) is $8_{17}$. Figure \ref{24tiles-8crossings-1}(q) is $8_{18}$. Not all of these have  tile number 24. We already know $8_{1}$, $8_{2}$, $8_{4}$, $8_{7}$, $8_{8}$, and $8_{13}$ have  tile number 22. Each of the others have  tile number 24. The non-alternating knots $8_{19}$, $8_{20}$, and $8_{21}$ are obtained from choosing non-alternating crossings in a few of these. Those pictured in the table of knots come from the layout in Figure \ref{24tiles-8crossings-1}(p). Mosaics for all of these are given in the table of knots in Section \ref{table_of_knots}. The only knots with crossing number 8 that we have not yet found are $8_{6}$ and $8_{15}$, and now we know that they cannot be represented with eight crossings and twenty-four non-blank tiles.

We now turn our attention to mosaics with nine crossings. Just as before, we look at all possible placements, up to symmetry, of nine crossings, eliminate any composite knots, unnecessary loops, links and any duplicate layouts that are equivalent to others via simple mosaic planar isotopy moves. In the end, we get seven possible layouts, which are shown in Figure \ref{24tiles-9crossings-1}. Choosing specific crossings for each layout, in order, so that the knots are alternating, we obtain the seven knots $9_8$, $9_{11}$, $9_{14}$, $9_{17}$, $9_{23}$, $9_{27}$, and $9_{31}$, all of which have  tile number 24. If we look at all possible choices for non-alternating crossings, the only knot with  tile number 24 that arises but did not show up with only eight crossing tiles is the $8_6$ knot, whose knot mosaic in the table of knots comes from the layout in Figure \ref{24tiles-9crossings-1}(a). All other prime knots that arise using non-alternating crossings have been exhibited as a minimally space-efficient mosaic with fewer crossings or fewer non-blank tiles.

\begin{figure}[h]
  \centering
  \includegraphics{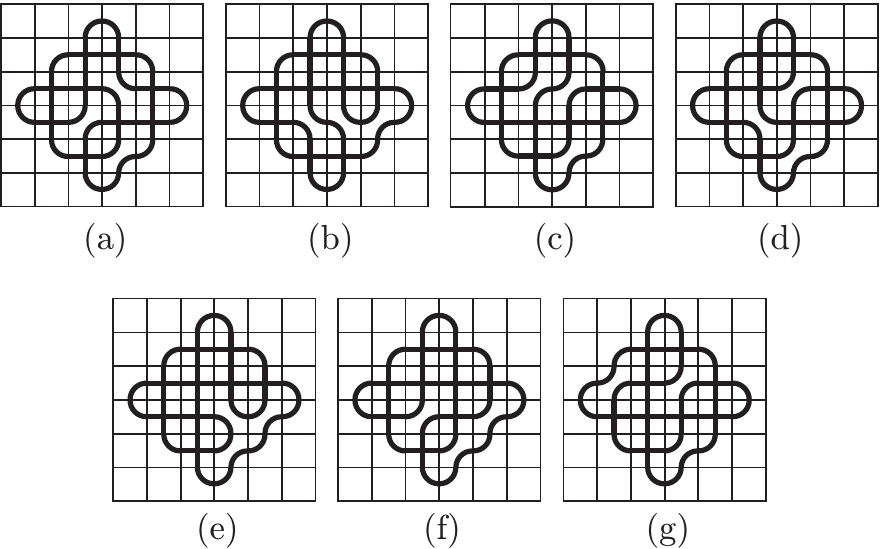}\\
  \caption{Only possible layouts, after elimination, with nine crossing tiles for a prime knot with  tile number 24.}
  \label{24tiles-9crossings-1}
\end{figure}

Now we do the same for ten crossings. Again, we observe all possible placements of ten crossings on the third mosaic in Figure \ref{tile-numbers-1}, and after eliminating any links and duplicate layouts up to reflection, rotation, or equivalencies via simple mosaic planar isotopy moves, we end up with five possible layouts, shown in Figure \ref{24tiles-10crossings-1}.

We begin with Figure \ref{24tiles-10crossings-1}(a). Choosing specific crossings so that the knot is alternating, we obtain the $10_{116}$ knot. If we look at all possible choices for non-alternating crossings, the only prime knots that we get with  tile number 24 are the non-alternating knots $10_{124}$, $10_{125}$, $10_{141}$, $10_{143}$, $10_{155}$, and $10_{159}$. We do the same with Figure \ref{24tiles-10crossings-1}(b) and get the alternating knot $10_{100}$. For the non-alternating choices, we get almost all of the same ones we just obtained, but we do not get any new additions to our list of knots. For Figure \ref{24tiles-10crossings-1}(c), with alternating crossings we get $10_{41}$, and with non-alternating crossings we get $9_{19}$ and $9_{21}$ as the only new additions to our list. Neither of these came from considering only nine crossings. Now we observe the mosaic in Figure \ref{24tiles-10crossings-1}(d). By alternating the crossings, we obtain $10_{44}$, and by using non-alternating crossings, the only new additions to our list are $9_{12}$ and $9_{26}$. Finally, we end with Figure \ref{24tiles-10crossings-1}(e). Assigning alternating crossings, we get $10_{85}$, and assigning non-alternating crossings, we get $10_{126}$, $10_{127}$, and $10_{148}$.

\begin{figure}[h]
  \centering
  \includegraphics{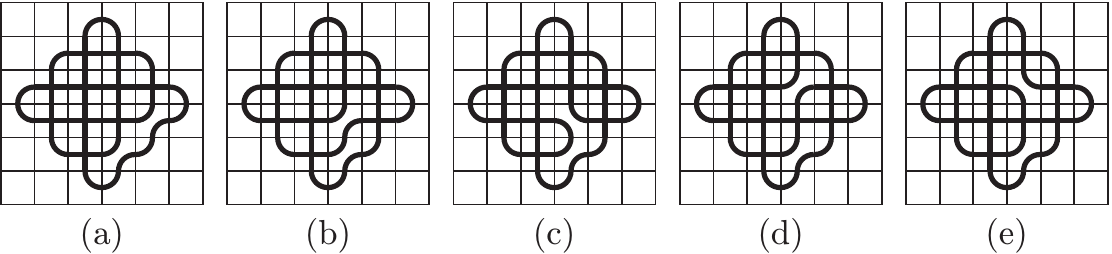}\\
  \caption{Only possible layouts, after elimination, with ten crossing tiles for a prime knot with  tile number 24.}
  \label{24tiles-10crossings-1}
\end{figure}

Finally, we can place eleven or twelve crossing tiles into the layout with twenty-four non-blank tiles, but the space-efficient results will always be a link with more than one component. Therefore, no minimally space-efficient prime knot mosaics arise from this consideration. We have considered every possible placement of crossing tiles on the third layout in Figure \ref{tile-numbers-1} and have found every possible prime knot with  tile number 24 and eight or more crossings, and they are exactly those listed in the theorem. Minimally space-efficient mosaics for all of these knots are given in the table of knots in Section \ref{table_of_knots}. All other prime knots with crossing number at least eight must have  tile number larger than 24.
\end{proof}

We now know precisely which prime knots have  tile number less than or equal to 24, and we are ready to determine which prime knots with mosaic number 6 have  tile number 27. We see our first occurrence of knots with crossing number larger than 10, and we use the Dowker-Thistlethwaite name of the knot.

\begin{theorem}\label{t=27}
The only prime knots $K$ with mosaic number $m(K)=6$, tile number $t(K)=27$, and minimal mosaic tile number $t_M(K)=27$ are:
\begin{enumerate}
  \item $8_{15}$
  \item $9_{1}$, $9_{2}$, $9_{3}$, $9_{4}$, $9_{7}$, $9_{9}$, $9_{13}$, $9_{24}$, $9_{28}$, $9_{37}$, $9_{46}$, $9_{48}$,
  \item $10_{1}$, $10_{2}$, $10_{3}$, $10_{4}$, $10_{12}$, $10_{22}$, $10_{28}$, $10_{34}$, $10_{63}$, $10_{65}$, $10_{66}$, $10_{75}$, $10_{78}$, $10_{140}$, $10_{142}$, $10_{144}$,
  \item \ka{11}{107}, \ka{11}{140}, and \ka{11}{343}.
\end{enumerate}
\end{theorem}

Notice that this theorem is only referring to prime knots with mosaic number 6. There are certainly prime knots with tile number 27 and mosaic number 7 that are not included in this theorem. Also, the requirement that the tile number equals the minimal mosaic tile number is necessary here. As far as we know now (and will verify below), there are knots with mosaic number 6 and tile number 27 which have minimal mosaic number 32. Some of these are listed in the next theorem. Finally, notice that up to this point we have determined the tile number for every prime knot with crossing number 8 or less.

Again we claim that the minimally space-efficient mosaics for $9_3$, $9_4$, $9_{13}$, $9_{37}$, $9_{46}$, and $9_{48}$ must have ten crossing tiles. The minimally space-efficient mosaics for $9_{7}$, $9_{9}$, and $9_{24}$ must have eleven crossing tiles. None of the possible minimally space-efficient knot mosaics with exactly nine crossing tiles produce these knots. Similarly, the minimally space-efficient mosaics for $10_{1}$, $10_{3}$, $10_{12}$, $10_{22}$, $10_{34}$, $10_{63}$, $10_{65}$, $10_{78}$, $10_{140}$, $10_{142}$, and $10_{144}$ require eleven crossing tiles.

\begin{proof}  Similar to what we did in the proof of Theorem \ref{t=22}, we search for all of the prime knots that mosaic number 6 and tile number 27, which have a space-efficient mosaic as depicted in the fourth layout of Figure \ref{tile-numbers-1}. We simply build this layout using the $3 \times 3$ building blocks that result from the observations at the beginning of this section and shown again in Figure \ref{building-blocks}. We then choose specific crossing types for each crossing tile and see what knots we get.

\begin{figure}[h]
  \centering
  \includegraphics{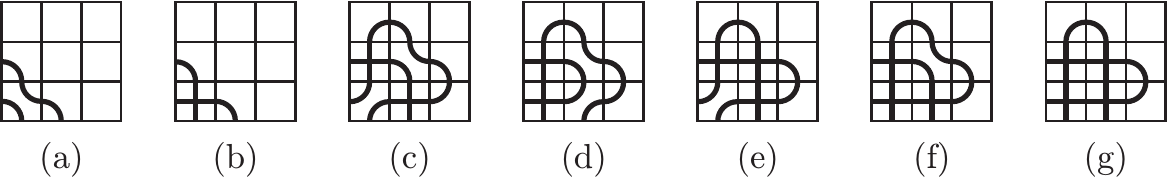}\\
  \caption{The seven building blocks resulting from the observations at the beginning of this section.}
  \label{building-blocks}
\end{figure}

For bookkeeping purposes, we note that the knot $8_{15}$ has  tile number 27, and this is the only knot with crossing number 8 for which we have not previously found the  tile number. A minimally space-efficient mosaic for it is included in the table of knots in Section \ref{table_of_knots}. We now know the  tile number for every prime knot with crossing number 8 or less, and from here we restrict our search to mosaics with nine or more crossing tiles.

Before we get started placing crossing tiles, we make a few more simple observations that apply to this particular case and help us reduce the number of possible configurations. Observe that if we place a partially filled building block with no crossing adjacent to the filled building block with two crossing tiles in positions $I^3$ and $I^4$ depicted in Figure \ref{building-blocks}(c), the resulting mosaic will always reduce to a mosaic with tile number 22. The same result holds if the two blocks are not adjacent and one of the adjacent blocks is the filled building block with three crossings depicted in Figure \ref{building-blocks}(e). The mosaics in Figure \ref{observ1-27} exhibit these scenarios. The same result also holds if the partially filled building block with one crossing is combined with two of the filled building blocks with two crossing tiles shown in Figure \ref{building-blocks}(c). Depending on the placement of these two filled blocks, the result will be equivalent to either Figure \ref{observ1-27}(a) or Figure \ref{observ1-27}(b) via a simple mosaic planar isotopy move that shifts the crossing in the partially filled block to another block.

\begin{figure}[h]
  \centering
  \includegraphics{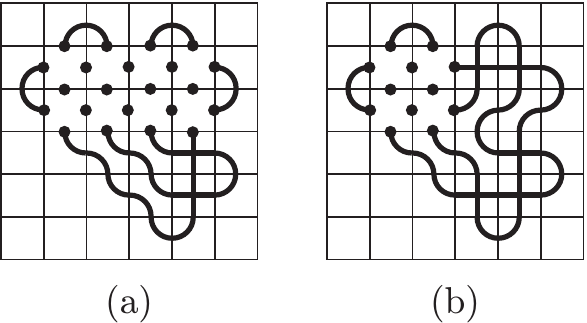}\\
  \caption{These two mosaics are not minimally space-efficient.}
  \label{observ1-27}
\end{figure}

First, we consider nine crossing tiles with the above observations in mind, together with the observations at the beginning of this section. Up to symmetry, there are only nine possible configurations of the building blocks after we eliminate the links, duplicate layouts that are equivalent to others via simple mosaic planar isotopy moves, and any mosaics for which the tile number can easily be reduced by a simple mosaic planar isotopy move. They are shown in Figure \ref{27tiles-9crossings-1}. Not all of these will result in distinct knots, and in several cases it is not difficult to see that they will result in the same knot. However, we include all of them here because they differ by more than just symmetries or a simple mosaic planar isotopy move.

\begin{figure}[h]
  \centering
  \includegraphics{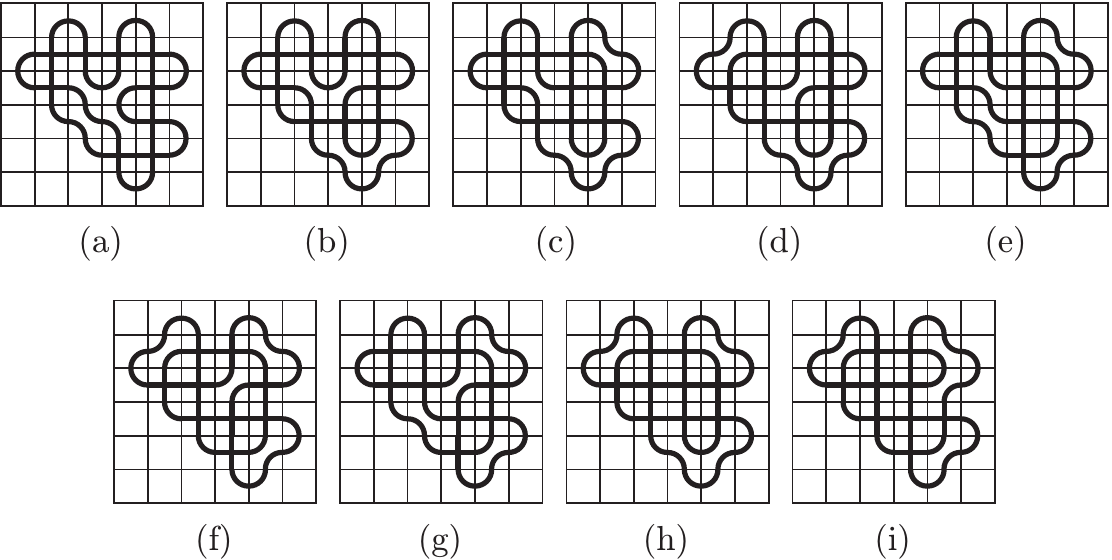}\\
  \caption{Only possible layouts, after elimination, with nine crossing tiles for a prime knot with  tile number 27.}
  \label{27tiles-9crossings-1}
\end{figure}

Choosing specific crossings so that the knots are alternating, we obtain only seven distinct knots. The only ones with  tile number 27 are Figure \ref{27tiles-9crossings-1}(a), which gives the $9_{1}$ knot, Figure \ref{27tiles-9crossings-1}(b), which gives us $9_{2}$, and Figures \ref{27tiles-9crossings-1}(h) and (i), which give us $9_{28}$. Each of the remaining layouts give knots with  tile number less than 27. In particular, Figures \ref{27tiles-9crossings-1}(c) and (d) are $9_{8}$, Figures \ref{27tiles-9crossings-1}(e) and (f) are $9_{17}$, and Figure \ref{27tiles-9crossings-1}(g) is $9_{20}$. None of these configurations give non-alternating knots with crossing number 9.

Second, we do the same for ten crossings. Again, we use the building blocks to build all possible configurations of the crossings, and up to symmetry, there are only six possibilities after eliminating any links and duplicate layouts that are equivalent via simple mosaic planar isotopy moves. These are shown in Figure \ref{27tiles-10crossings-1}.

\begin{figure}[h]
  \centering
  \includegraphics{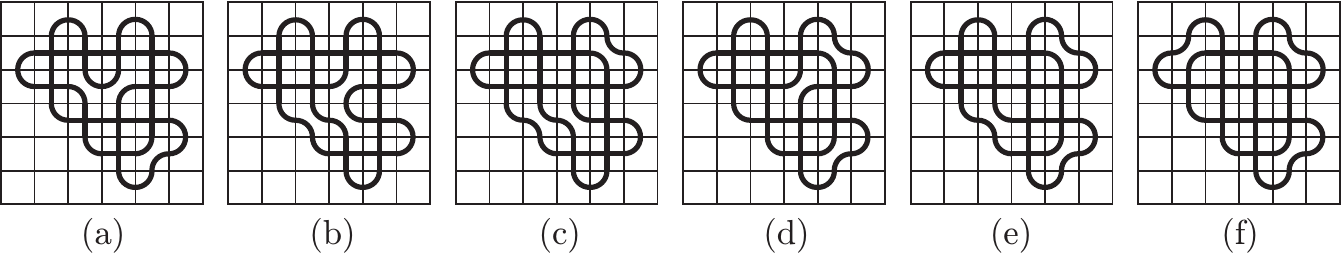}\\
  \caption{Only possible layouts, after elimination, with ten crossing tiles for a prime knot with  tile number 27.}
  \label{27tiles-10crossings-1}
\end{figure}

Choosing specific crossings so that the knots are alternating, we obtain only five distinct knots, all of which have  tile number 27. In particular, Figure \ref{27tiles-10crossings-1}(a) becomes the $10_{2}$ knot, Figure \ref{27tiles-10crossings-1}(b) becomes $10_{4}$, Figures \ref{27tiles-10crossings-1}(c) and (d) become $10_{28}$, Figure \ref{27tiles-10crossings-1}(e) becomes $10_{66}$, and Figure \ref{27tiles-10crossings-1}(f) becomes $10_{75}$. Choosing non-alternating crossings, we also get some knots with crossing number nine, but we do not obtain any non-alternating knots with crossing number ten. We can get $9_{3}$ from Figure \ref{27tiles-10crossings-1}(a), $9_{4}$ from Figure \ref{27tiles-10crossings-1}(b), $9_{13}$ from Figure \ref{27tiles-10crossings-1}(c), and $9_{37}$, $9_{46}$, and $9_{48}$ from Figure \ref{27tiles-10crossings-1}(f). All other knots that are obtained by considering non-alternating crossings can be drawn with fewer crossings or a lower tile number.

Third, we consider the case where the mosaic has eleven crossing tiles. In this instance, we end up with the five possible layouts shown in Figure \ref{27tiles-11crossings-1}, and again, not all of these are distinct. Choosing alternating crossing in each layout results in three distinct knots with crossing number eleven. Figures \ref{27tiles-11crossings-1}(a) and (b) become \ka{11}{107}, Figures \ref{27tiles-11crossings-1}(c) and (d) become \ka{11}{140}, and Figure \ref{27tiles-11crossings-1}(e) becomes \ka{11}{343}. (Note that, for knots with crossing number greater than ten, we are using the Dowker-Thistlethwaite name of the knot.) Choosing non-alternating crossings in each of the layouts results in several knots with crossing number nine or ten. In particular, we can obtain the knots $9_{24}$, $10_{63}$, $10_{65}$, $10_{78}$, $10_{140}$, $10_{142}$, and $10_{144}$ from Figure \ref{27tiles-11crossings-1}(a). We can obtain
$9_{7}$, $9_{9}$, $10_{12}$, $10_{22}$, and $10_{34}$ from Figure \ref{27tiles-11crossings-1}(c). And we can obtain $10_{1}$ and $10_{3}$ from Figure \ref{27tiles-11crossings-1}(e). All of these are shown in the table of knots in Section \ref{table_of_knots}. All other knots that are obtained by considering non-alternating crossings can be drawn with fewer crossings or a lower tile number.

\begin{figure}[h]
  \centering
  \includegraphics{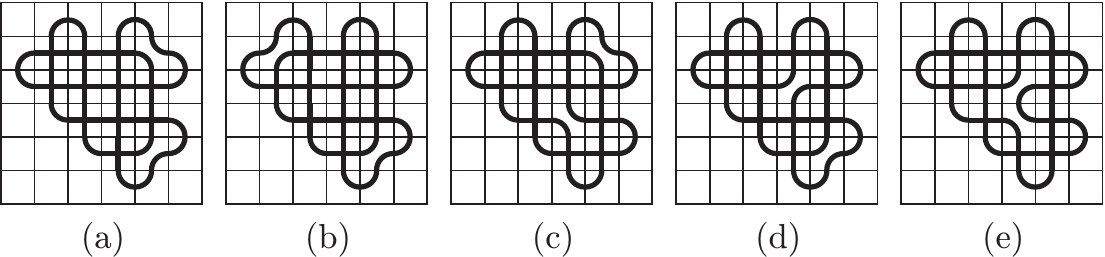}\\
  \caption{Only possible layouts, after elimination, with eleven crossing tiles for a prime knot with  tile number 27.}
  \label{27tiles-11crossings-1}
\end{figure}

Finally, by Observation \ref{observation4} we do not need to consider twelve or more crossing tiles in this layout, as no minimally space-efficient prime knot mosaics arise from this consideration. We have considered every possible placement of nine or more crossing tiles on the fourth layout in Figure \ref{tile-numbers-1} and have found every possible prime knot with mosaic number 6 and  tile number 27. They are exactly those listed in the theorem. All other prime knots with crossing number at least nine and mosaic number 6 must have minimal mosaic tile number 32.
\end{proof}

Now we know the  tile number for every prime knot with crossing number less than or equal to 8. Theorems \ref{t=22}, \ref{t=24}, and \ref{t=27} tell us the  tile number of some of the prime knots with crossing number 9, 10, and 11. Furthermore, we know that all other prime knots with mosaic number 6 must have minimal mosaic tile number 32 but not necessarily tile number 32. One problem that complicates the next step is that, as of the writing of this paper, we do knot know the mosaic number of all prime knots with crossing number 9 or more. That is, we do not know all prime knots with mosaic number 6. For this reason, we need to go through the same process as we did in the preceding proofs to determine which prime knots have mosaic number 6 and minimal mosaic tile number 32. By doing this, we will also be able to determine which prime knots have mosaic number greater than 6. The good news is that this is the final step in determining which prime knots have mosaic number 6 or less and determining the tile number or minimal mosaic tile numbers of all of these.

\begin{theorem}\label{t=32}
The only prime knots $K$ with mosaic number $m(K)=6$ and minimal mosaic tile number $t_M(K)=32$ are:
\begin{enumerate}
  \item $9_{10}$, $9_{16}$, $9_{35}$,
  \item $10_{11}$, $10_{20}$, $10_{21}$, $10_{61}$, $10_{62}$, $10_{64}$, $10_{74}$, $10_{76}$, $10_{77}$, $10_{139}$,
  \item \ka{11}{43}, \ka{11}{44}, \ka{11}{46}, \ka{11}{47}, \ka{11}{58}, \ka{11}{59}, \ka{11}{106}, \ka{11}{139}, \ka{11}{165}, \ka{11}{166}, \ka{11}{179}, \ka{11}{181}, \ka{11}{246}, \ka{11}{247}, \ka{11}{339}, \ka{11}{340}, \ka{11}{341}, \ka{11}{342}, \ka{11}{364}, \ka{11}{367},
  \item \kn{11}{71}, \kn{11}{72}, \kn{11}{73}, \kn{11}{74}, \kn{11}{75}, \kn{11}{76}, \kn{11}{77}, \kn{11}{78},
  \item  \ka{12}{119},\ka{12}{165}, \ka{12}{169}, \ka{12}{373}, \ka{12}{376}, \ka{12}{379}, \ka{12}{380}, \ka{12}{444},\ka{12}{503}, \ka{12}{722}, \ka{12}{803}, \ka{12}{1148}, \ka{12}{1149}, \ka{12}{1166},
  \item \ka{13}{1230}, \ka{13}{1236}, \ka{13}{1461}, \ka{13}{4573}
  \item \kn{13}{2399}, \kn{13}{2400}, \kn{13}{2401}, \kn{13}{2402}, \kn{13}{2403}.
\end{enumerate}
\end{theorem}

Notice again our restriction to prime knots with mosaic number 6. Additionally, notice that this theorem only refers to the minimal mosaic tile number of the knot, not the tile number. Again, this is because we only know that these two numbers are equal when they are less than or equal to 27. Some of these knots may have (and actually do have) tile number less than 32.

We claim that the minimally space-efficient mosaics for $9_{10}$, $9_{16}$, $10_{20}$, $10_{21}$, and $10_{77}$ need eleven crossing tiles. The minimally space-efficient mosaics for $9_{35}$, $10_{11}$, $10_{62}$, $10_{64}$, $10_{74}$, $10_{139}$, \ka{11}{106}, \ka{11}{139}, \ka{11}{166}, \ka{11}{181}, \ka{11}{341}, \ka{11}{342}, and \ka{11}{364} need twelve crossing tiles. And the minimally space-efficient mosaics for $10_{61}$, $10_{76}$, \ka{11}{44}, \ka{11}{47}, \ka{11}{58}, \kn{11}{76}, \kn{11}{77}, \kn{11}{78}, \ka{11}{165}, \ka{11}{246}, \ka{11}{339}, \ka{11}{340}, \ka{12}{119}, \ka{12}{165}, \ka{12}{169},  \ka{12}{376}, \ka{12}{379}, \ka{12}{444}, \ka{12}{803}, \ka{12}{1148}, and \ka{12}{1166} need thirteen crossing tiles.

\begin{proof}  We simply go through the same process that we did in the previous proof. We search for all of the prime knots that have mosaic number 6 and minimal mosaic tile number 32. Whatever prime knots that do not show up in this process and that we have not previously determined the  tile number for must have mosaic number greater than 6. We know from Theorem \ref{tile-numbers} that any prime knot with mosaic number 6 and minimal mosaic tile number 32 has a space-efficient mosaic with the fifth and final layout shown in Figure \ref{tile-numbers-1}.

As we have done several times previously, we use the building blocks to achieve all possible configurations, up to symmetry, of nine or more crossings within this mosaic. For this particular layout, we can only use the filled blocks, not the partially filled blocks. We can eliminate any layouts that do not meet the requirements of the observations, any multi-component links, any duplicate layouts that are equivalent to others via simple mosaic planar isotopy moves, and any mosaics for which the tile number can easily be reduced by a simple mosaic planar isotopy move.

First, in the case of nine crossings, after we eliminate the unnecessary layouts we end up with only one possibility, and it is shown in Figure \ref{32tiles-9crossings-1}. However, once we choose specific crossings in an alternating fashion, it is the knot $9_{8}$, which has  tile number 24. Nothing new arises from considering non-alternating crossings either.

\begin{figure}[h]
  \centering
  \includegraphics{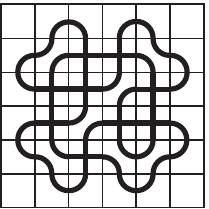}\\
  \caption{Only possible layout, after elimination, with nine crossing tiles for a prime knot with minimal mosaic tile number 32.}
  \label{32tiles-9crossings-1}
\end{figure}

Second, we do the same for ten crossings, and we end up with five possible layouts, shown in Figure \ref{32tiles-10crossings-1}. Choosing alternating crossings in each one, we again fail to get any prime knots with minimal mosaic tile number 32. Figure \ref{32tiles-10crossings-1}(a) is $10_{1}$, Figure \ref{32tiles-10crossings-1}(b) and (c) are $10_{34}$, and Figures \ref{32tiles-10crossings-1}(d) and (e) are $10_{78}$. Nothing new arises from considering non-alternating crossings either.

\begin{figure}[h]
  \centering
  \includegraphics{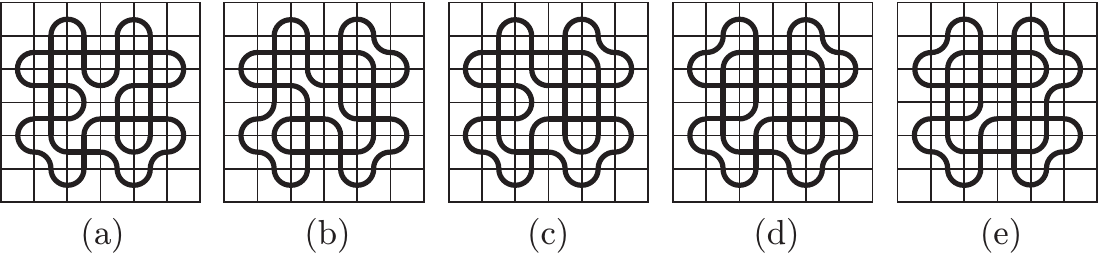}\\
  \caption{Only possible layouts, after elimination, with ten crossing tiles for a prime knot with minimal mosaic tile number 32.}
  \label{32tiles-10crossings-1}
\end{figure}

Third, we consider the case where the mosaic has eleven crossing tiles. In this instance, we end up with the ten possible layouts shown in Figure \ref{32tiles-11crossings-1}. With alternating crossings, the first layout is \ka{11}{140}, which we already known has  tile number 27. The remaining layouts, given alternating crossings, lead to six distinct knots with minimal mosaic tile number 32, and with non-alternating crossings we get ten additional knots that have minimal mosaic tile number 32. In particular, Figure \ref{32tiles-11crossings-1}(b) with alternating crossings is \ka{11}{43} and with non-alternating crossings can be made into \kn{11}{71}, \kn{11}{72}, \kn{11}{73}, \kn{11}{74}, and \kn{11}{75}. Figures \ref{32tiles-11crossings-1}(c) and (d) are \ka{11}{46} when using alternating crossings and can be made into $9_{16}$ or $10_{77}$ with non-alternating crossings. Figures \ref{32tiles-11crossings-1}(e) and (f) are \ka{11}{59} when using alternating crossings and can be made into $10_{20}$ with non-alternating crossings. Figures \ref{32tiles-11crossings-1}(g) and (h) are \ka{11}{179} when using alternating crossings and can be made into $9_{10}$ or $10_{21}$ with non-alternating crossings. Figure \ref{32tiles-11crossings-1}(i) with alternating crossings is \ka{11}{247}, and Figure \ref{32tiles-11crossings-1}(j) with alternating crossings is \ka{11}{367}. Neither of these last two provide new knots to our list when considering non-alternating crossings.

\begin{figure}[h]
  \centering
  \includegraphics{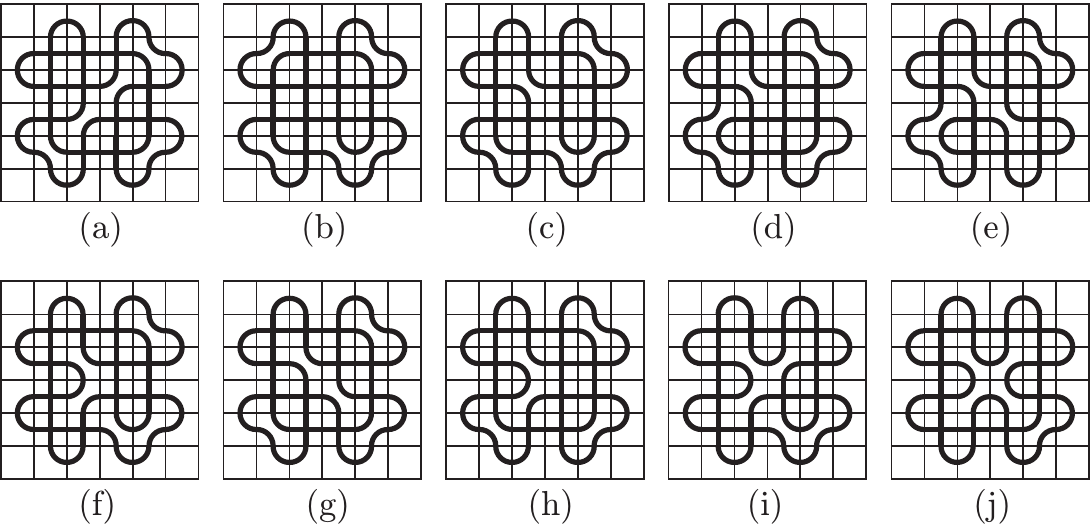}\\
  \caption{Only possible layouts, after elimination, with eleven crossing tiles for a prime knot with minimal mosaic tile number 32.}
  \label{32tiles-11crossings-1}
\end{figure}

Fourth, we consider the possibilities where the mosaic has twelve crossing tiles. In this case, we end up with the seven possible layouts shown in Figure \ref{32tiles-12crossings-1}. With alternating crossings, these layouts lead to five distinct knots with minimal mosaic tile number 32, and with non-alternating crossings we get thirteen additional knots that have minimal mosaic tile number 32. In particular, Figures \ref{32tiles-12crossings-1}(a) and (b) with alternating crossings are \ka{12}{373} and with non-alternating crossings can be made into $10_{62}$, $10_{64}$, $10_{139}$, \ka{11}{106}, or \ka{11}{139}. Figures \ref{32tiles-12crossings-1}(c) and (d) are \ka{12}{380} when using alternating crossings and can be made into $10_{11}$, \ka{11}{166}, or \ka{11}{341} with non-alternating crossings. Figure \ref{32tiles-12crossings-1}(e) is \ka{12}{503} when using alternating crossings and can be made into $9_{35}$, $10_{74}$, or \ka{11}{181} with non-alternating crossings. Figure \ref{32tiles-12crossings-1}(f) is \ka{12}{722} when using alternating crossings and can be made into \ka{11}{364} with non-alternating crossings. Figure \ref{32tiles-12crossings-1}(g) with alternating crossings is \ka{12}{1149} and with non-alternating crossings can be \ka{11}{342}.

\begin{figure}[h]
  \centering
  \includegraphics{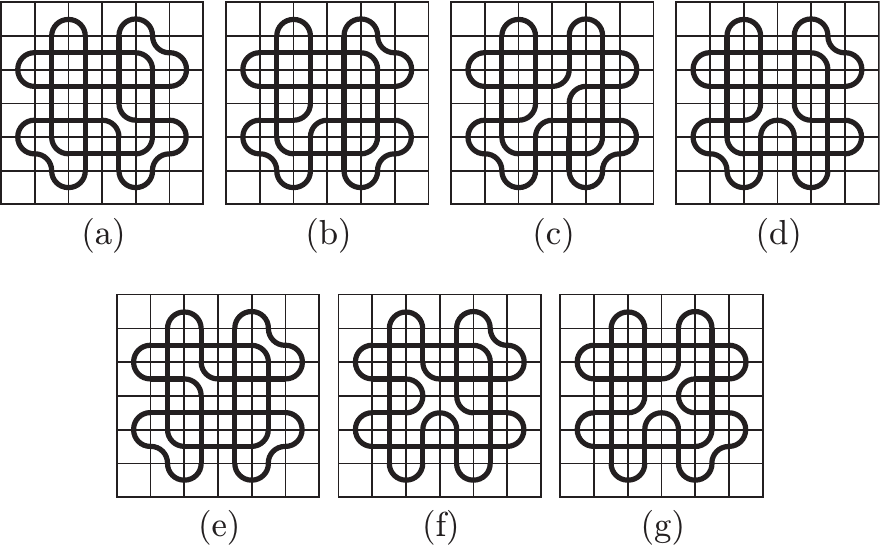}\\
  \caption{Only possible layouts, after elimination, with twelve crossing tiles for a prime knot with minimal mosaic tile number 32.}
  \label{32tiles-12crossings-1}
\end{figure}

Fifth, we consider what happens when we place thirteen crossing tiles on the mosaic. In this instance, we end up with the six possible layouts shown in Figure \ref{32tiles-13crossings-1}. With alternating crossings, the layouts lead to four distinct knots with minimal mosaic tile number 32, and with non-alternating crossings we get twenty-six additional knots that have minimal mosaic tile number 32. In particular, Figure \ref{32tiles-13crossings-1}(a) with alternating crossings is \ka{13}{1230} and with non-alternating crossings can be made into \ka{11}{44}, \ka{11}{47}, \kn{11}{76}, \kn{11}{77}, \kn{11}{78}, \ka{12}{119}, \kn{13}{2399}, \kn{13}{2400}, \kn{13}{2401}, \kn{13}{2402}, or \kn{13}{2403}. Figures \ref{32tiles-13crossings-1}(b) and (c) are \ka{13}{1236} when using alternating crossings and can be made into $10_{61}$, $10_{76}$, \ka{11}{58}, \ka{11}{165}, \ka{11}{340}, \ka{12}{165}, \ka{12}{376}, or \ka{12}{444} with non-alternating crossings. Figures \ref{32tiles-13crossings-1}(d) and (e) are \ka{13}{1461} when using alternating crossings and can be made into \ka{11}{246}, \ka{11}{339}, \ka{12}{169},  \ka{12}{379},  or \ka{12}{1148} with non-alternating crossings. Figure \ref{32tiles-13crossings-1}(f) is \ka{13}{4573} when using alternating crossings and can be made into \ka{12}{803} or \ka{12}{1166} with non-alternating crossings.

\begin{figure}[h]
  \centering
  \includegraphics{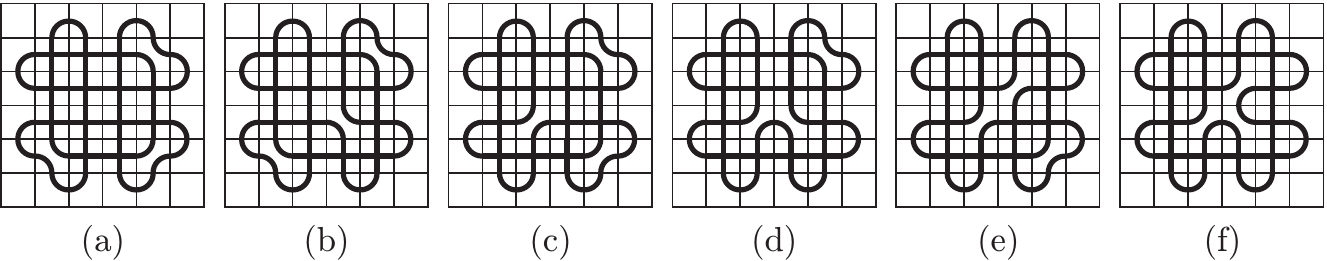}\\
  \caption{Only possible layouts, after elimination, with thirteen crossing tiles for a prime knot with minimal mosaic tile number 32.}
  \label{32tiles-13crossings-1}
\end{figure}

Finally, by Observation \ref{observation4}, we do not need to consider fourteen or more crossing tiles in this layout. We have considered every possible placement of nine or more crossing tiles on the final layout of Figure \ref{tile-numbers-1} and have found every possible prime knot with mosaic number 6 and minimal mosaic tile number 32.
\end{proof}

Because of the work we have completed, we now know every prime knot with mosaic number 6 or less. We also know the tile number or minimal mosaic tile number of each of these prime knots. In the table of knots in Section \ref{table_of_knots}, we provide minimally space-efficient knot mosaics for all of these. These preceding theorems lead us to the following interesting consequences.

\begin{corollary}\label{m>6} The prime knots with crossing number at least 9 not listed in Theorems \ref{t=22}, \ref{t=24}, \ref{t=27}, or \ref{t=32} have mosaic number 7 or higher.
\end{corollary}

\begin{theorem}\label{min-mos-tile-number} The tile number of a knot is not necessarily equal to the minimal mosaic tile number of a knot.
\end{theorem}

\begin{proof}
According to Theorem \ref{t=32}, the minimal mosaic tile number for $9_{10}$ is 32. However, on a $7$-mosaic, this knot can be represented using only 27 non-blank tiles, as depicted in Figure \ref{knot9_10}.
\begin{figure}[h]
  \centering
  \includegraphics{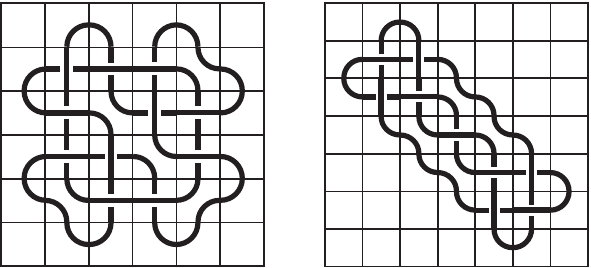}\\
  \caption{The $9_{10}$ knot represented as a minimally space-efficient 6-mosaic with minimal mosaic tile number 32 and as a space-efficient 7-mosaic with tile number 27.}
  \label{knot9_10}
\end{figure}
Also note that, as a 7-mosaic, this knot could be represented with only nine crossings, whereas eleven crossings were required to represent it as a 6-mosaic.
\end{proof}

\section{Table of Prime Knots}\label{table_of_knots}

We include a table of knots below, with an example of a minimally space-efficient knot mosaic for each prime knot for which we know the minimal mosaic tile number. For each knot mosaic, both the mosaic number and minimal mosaic tile number are realized, but the crossing number may not be realized. The tile number of the knot is realized in each mosaic unless the tile number of the mosaic is 32. If the knot mosaic is marked with an asterisk ({\Large $\ast$}) then the given mosaic has more crossing tiles than the crossing number for the represented knot, but it is the minimum number of crossing tiles needed in order for the minimal mosaic tile number to be realized.

\vspace{.2in}

\begin{center}

  \includegraphics{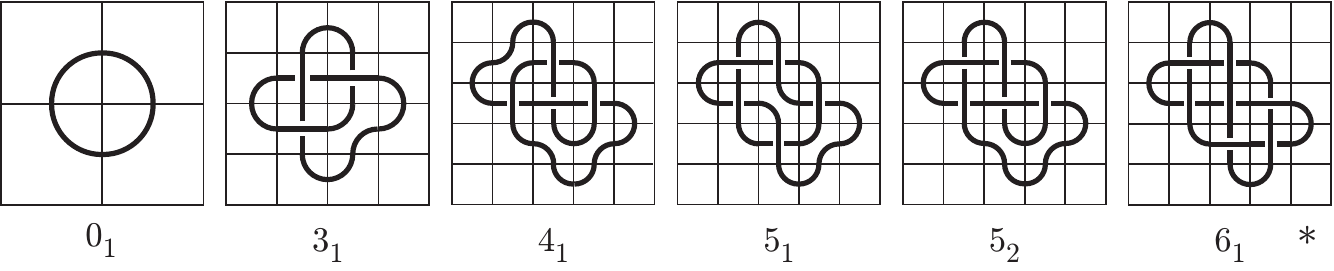} \vspace{.15in}

  \includegraphics{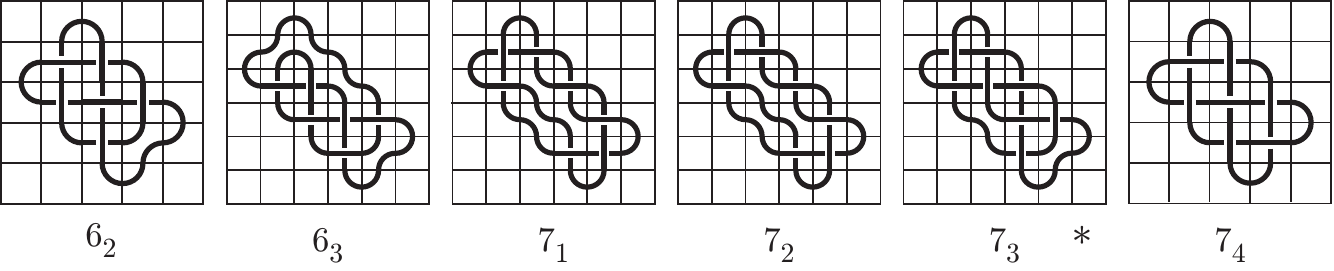} \vspace{.15in}

  \includegraphics{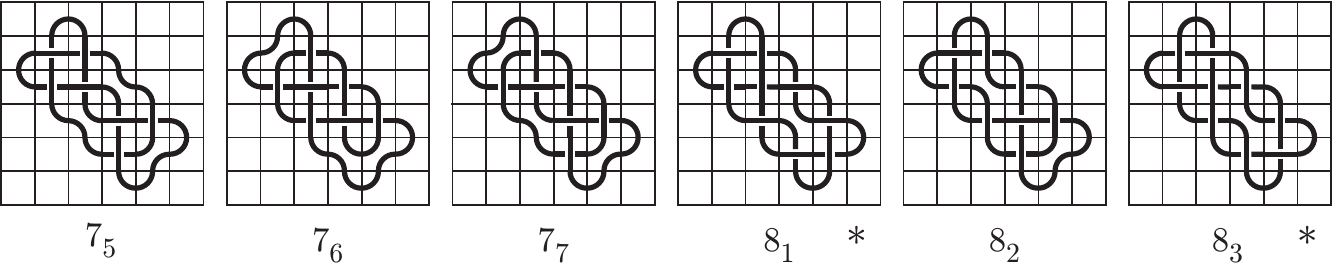} \vspace{.15in}

  \includegraphics{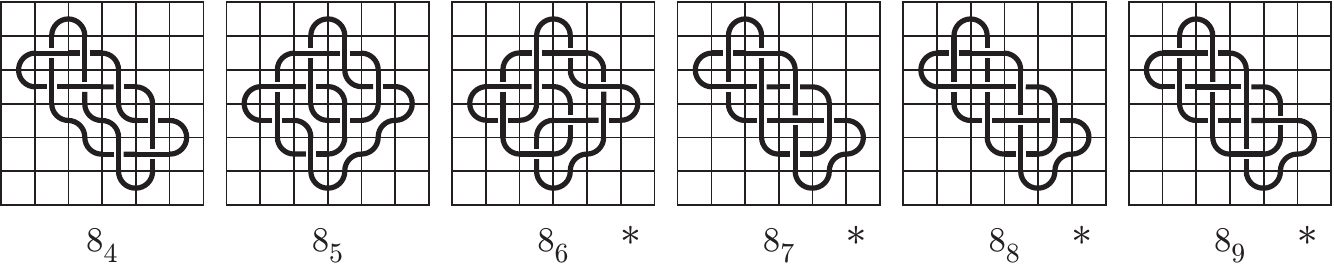} \vspace{.15in}

  \includegraphics{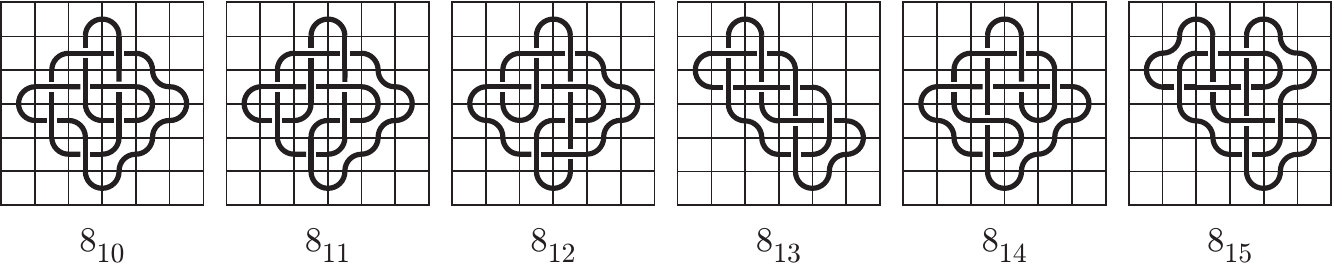} \vspace{.15in}

  \includegraphics{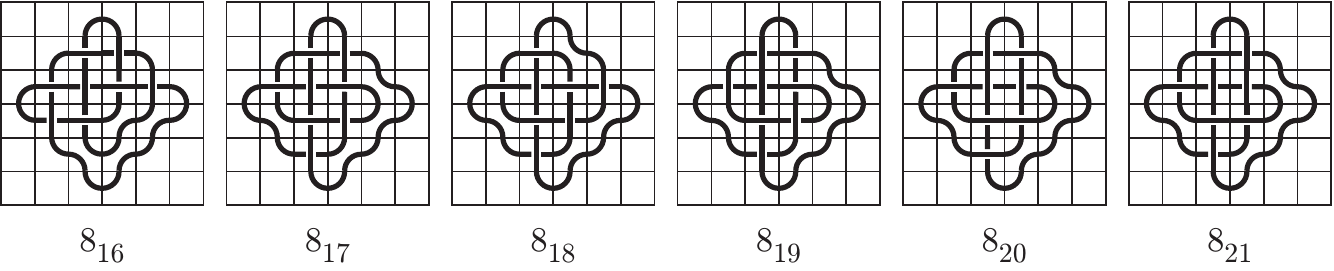} \vspace{.15in}

  \includegraphics{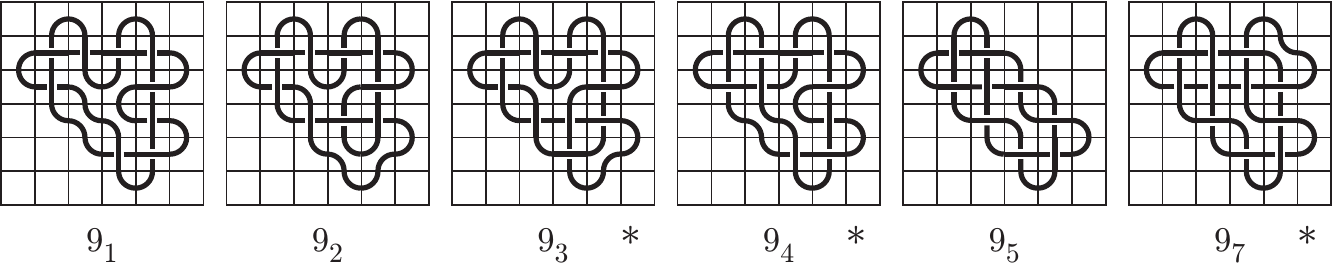} \vspace{.15in}

  \includegraphics{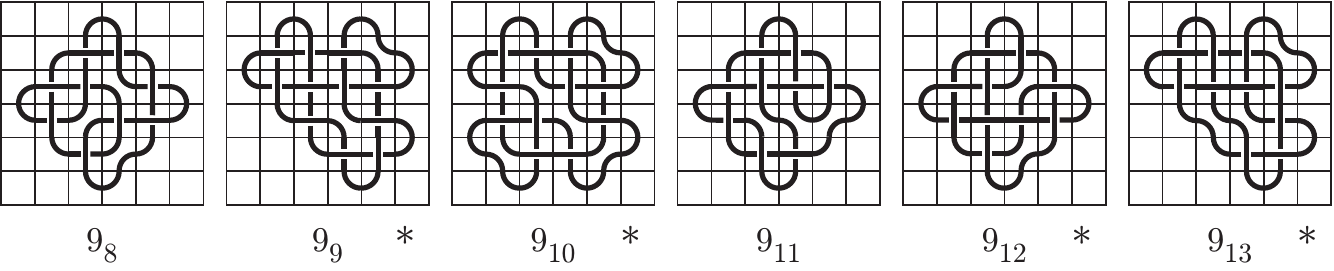} \vspace{.15in}

  \includegraphics{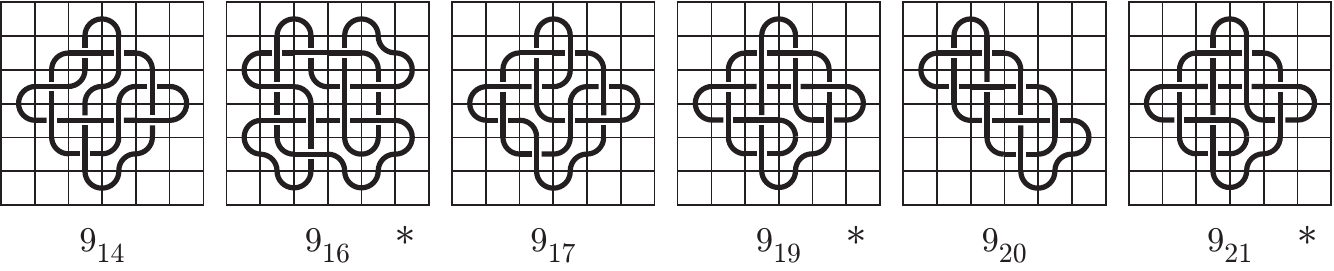} \vspace{.15in}

  \includegraphics{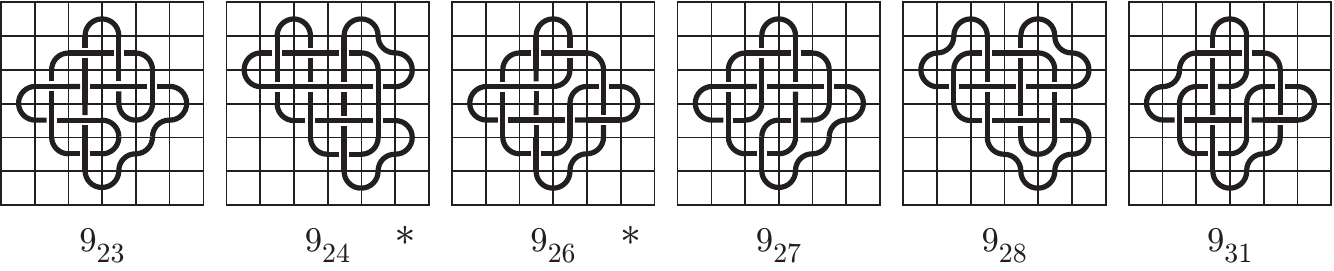} \vspace{.15in}

  \includegraphics{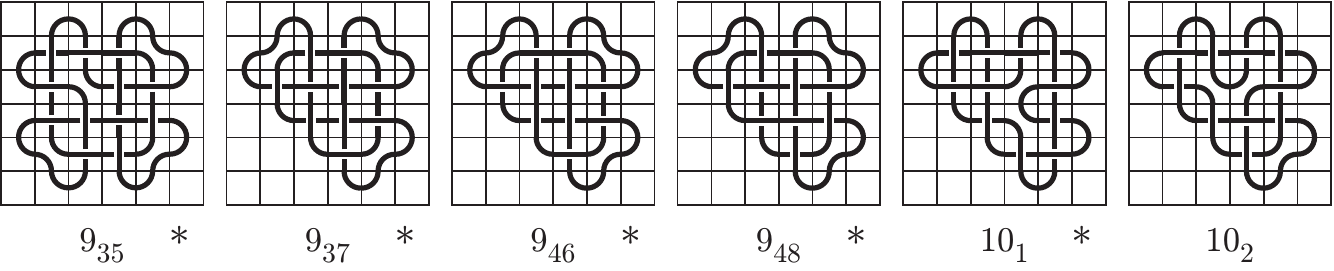} \vspace{.15in}

  \includegraphics{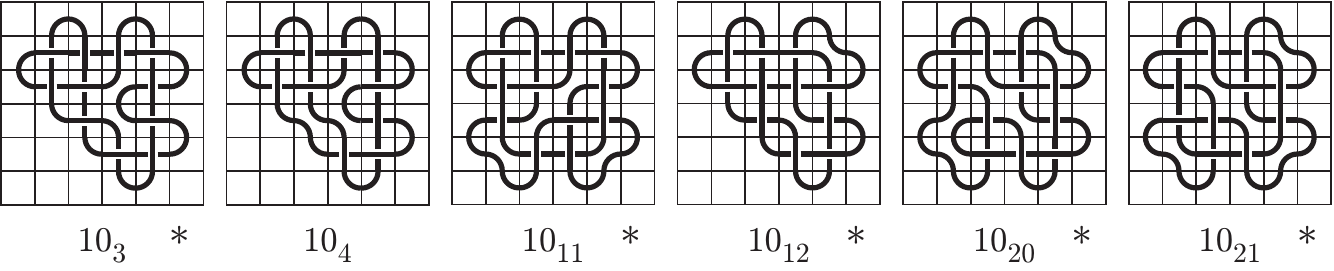} \vspace{.15in}

  \includegraphics{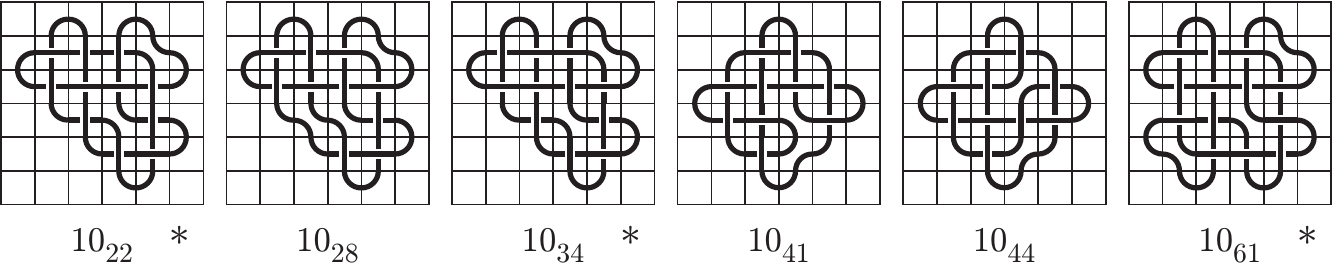} \vspace{.15in}

  \includegraphics{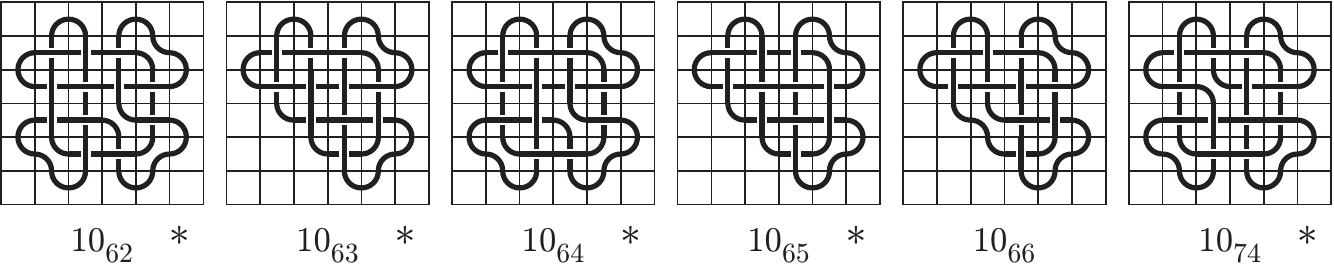} \vspace{.15in}

  \includegraphics{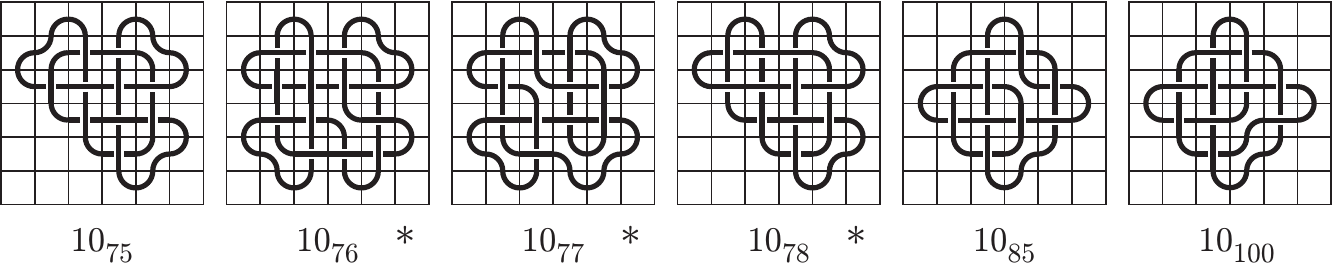} \vspace{.15in}

  \includegraphics{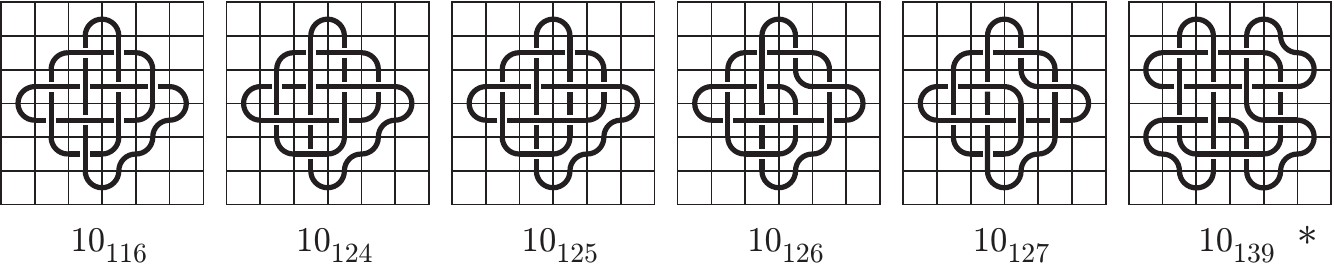} \vspace{.15in}

  \includegraphics{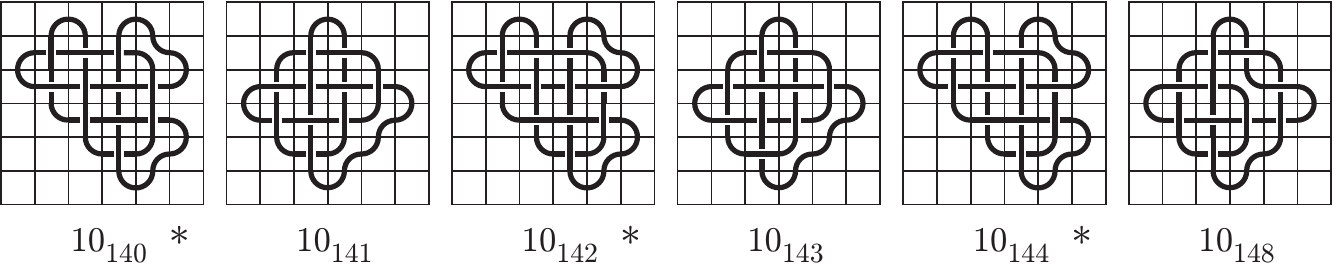} \vspace{.15in}

  \includegraphics{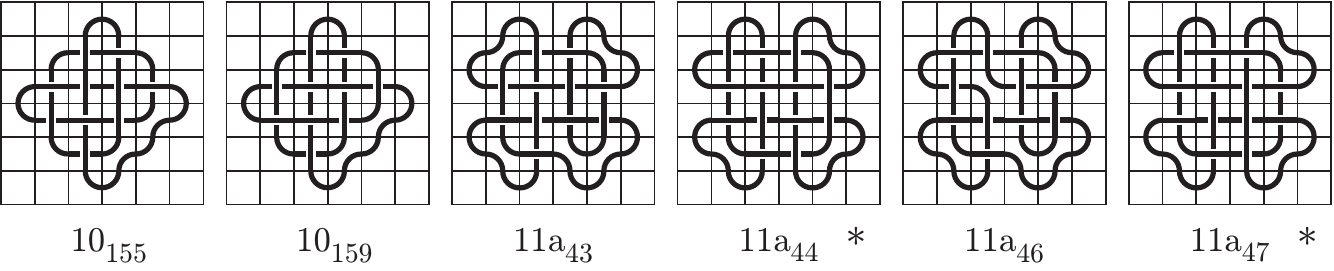} \vspace{.15in}

  \includegraphics{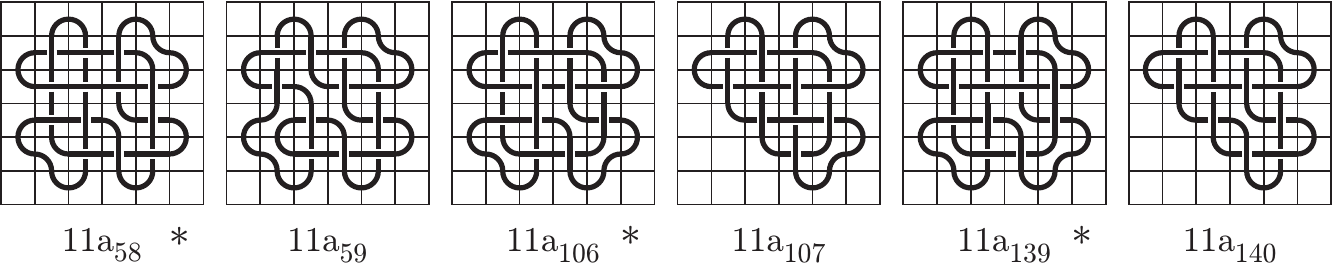} \vspace{.15in}

  \includegraphics{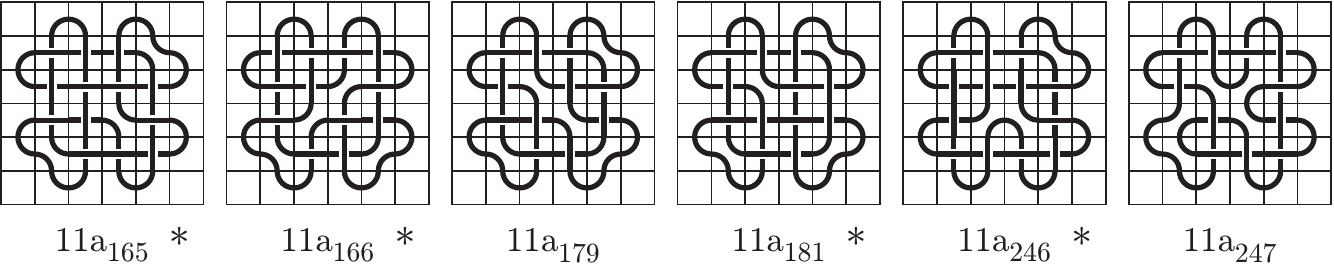} \vspace{.15in}

  \includegraphics{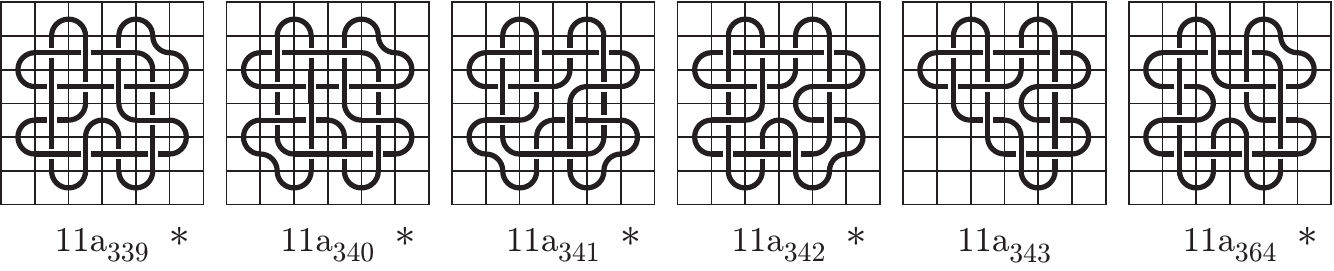} \vspace{.15in}

  \includegraphics{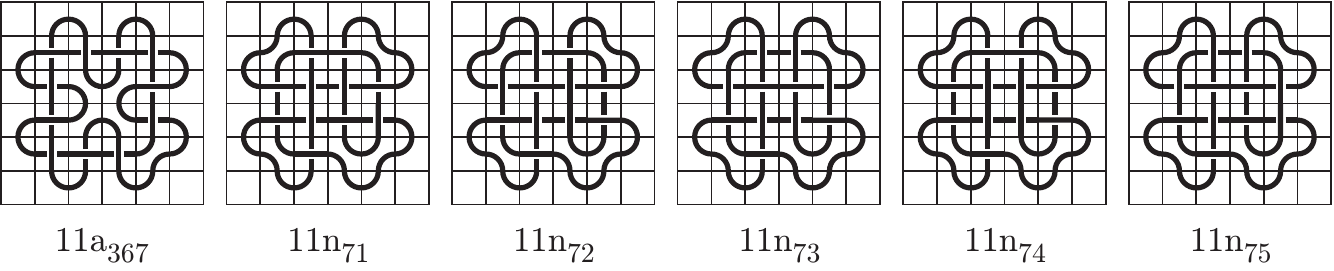} \vspace{.15in}

  \includegraphics{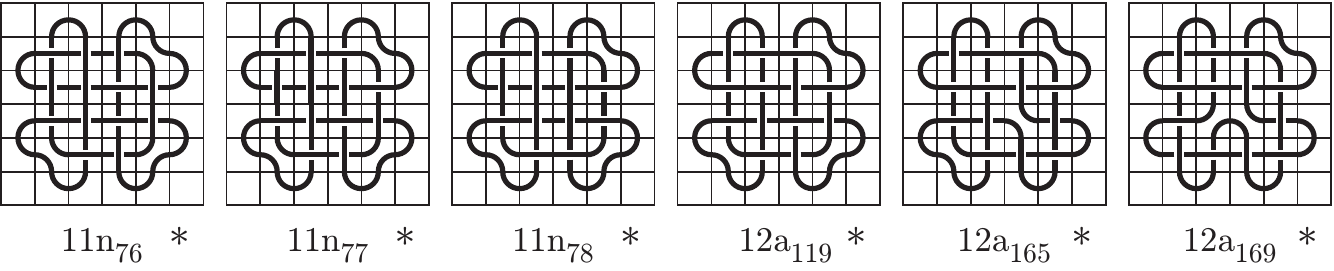} \vspace{.15in}

  \includegraphics{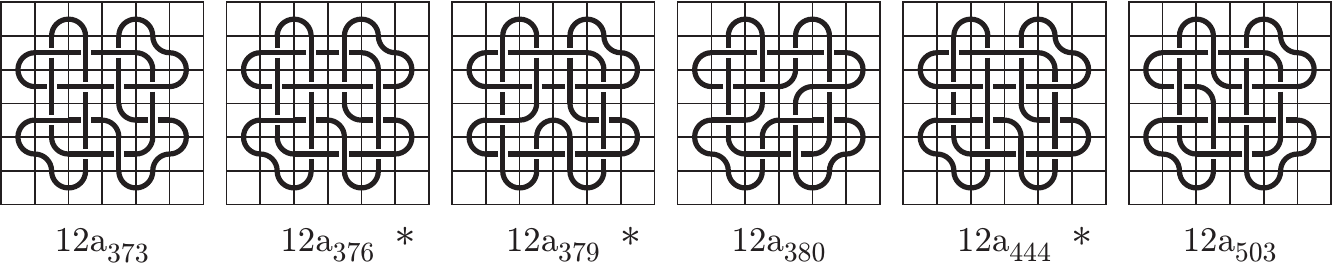} \vspace{.15in}

  \includegraphics{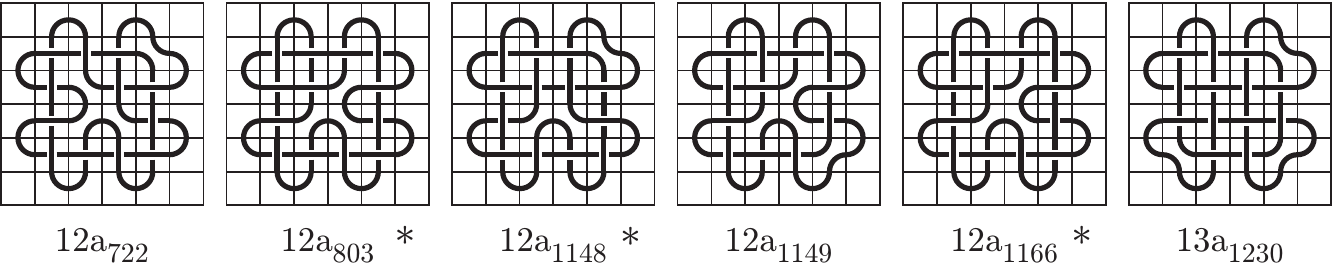} \vspace{.15in}

  \includegraphics{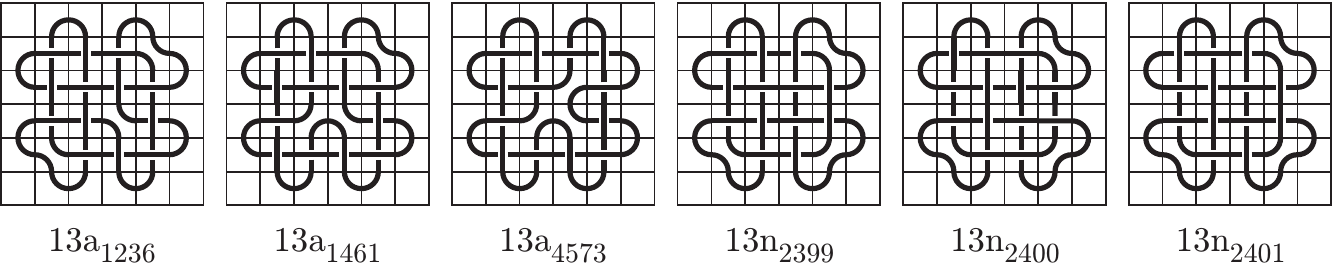} \vspace{.15in}

  \includegraphics{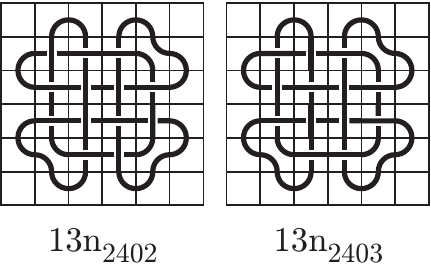}

\end{center}

\bibliographystyle{amsplain}
\bibliography{bibliography}
\addcontentsline{toc}{section}{\refname}

\end{document}